\documentclass[12pt]{amsart}
\usepackage{graphicx}
\usepackage[usenames]{color}
\usepackage{amsmath,amsfonts,amsthm,amssymb}
\usepackage{syntonly}
\usepackage[mathcal]{euscript}
\usepackage[colorlinks=true]{hyperref} 
\hypersetup{urlcolor=blue, citecolor=blue} 
\usepackage{color}

\usepackage{xypic} 
\parskip = \medskipamount

\def\HH{{\mathbb H}}

\def\QQ{{\mathbb Q}}
\def\RR{{\mathbb R}}

\def\ZZ{{\mathbb Z}}

\newtheorem{Def}{Definition}
\newtheorem{Lemma}{Lemma}
\newtheorem{Prop}{Proposition}
\newtheorem{Cor}{Corollary}
\newtheorem{Teo}{Theorem}

\newtheorem{Ex}{Example}

\newcommand{\lrar}{\longrightarrow}

\newcommand{\Hpp}{H^{\prime\prime}_{0}} 
\newcommand{\dt}{\partial_{t}} 

\begin{document}
	\title[Bilinear forms in the Jacobian module...]{	\small{
Bilinear forms in the Jacobian module\\ 
and binding
of $N$-spectral chains\\
of a hypersurface
with an isolated singularity}
}
	\author{Miguel Angel Dela-Rosa$^\dag$}
	\author{Xavier G\'omez-Mont$^\ddag$}
	\address{\text{$^\dag$} Divisi\'on Acad\'emica de Ciencias B\'asicas, CONACyT-UJAT\\
		Km 1, Carretera Cunduac\'an--Jalpa de M\'endez, Cunduac\'an, Tabasco, c.p. 86690, M\'exico}
	\email{$^\dag$madelarosaca@conacyt.mx}

\address{\text{$^\ddag$} Centro de Investigación en Matemáticas,  AP 402,  Guanajuato, Guanajuato, c.p. 36000, M\'exico} 
\email{$^\ddag$gmont@cimat.mx}	
	 
\date{\today}
\thanks{The first author was partially supported by CONACYT grant A1-S-47710. The second author was partially supported by CONACYT grant 286447.  
}

\maketitle
 
\begin{abstract} 
Using his deep and beautiful idea of cutting with a Hyperplane, Lefschetz explained how the homology groups of a projective smooth variety could be constructed from basic pieces, that he called primitive homology. This idea can be applied every time we have a vector space with a biliner  form (i.e. homology with cup product) and a $\pm$-symmetric nilpotent operator (i.e. cutting with a hyperplane). We will illustrate this in the context of Singularity Theory: A germ of an isolated singular point of a hypersurface defined by $f=0$. We begin with  the algebraic setting in the Jacobian (or Milnor)  Algebra of the singularity, with Grothendieck pairing as bilinear form and  multiplication by $f$ as a symmetric nilpotent operator. We continue in the topological setting of vanishing cohomology with  bilinear form induced from cup product and as nilpotent map the logarithm of the unipotent map of the monodromy as an anti-symmetric operator.
We then show how these 2 very different settings are tied up using the Brieskorn lattice as a D-module, on using results of Brieskorn (1970), A. Varchenko (1980s),  M. Saito (1989), and C. Hertling (1999, 2004, 2005), inducing a Polarized Mixed Hodge structure at the singularity (Steenbrink, 1976) and bringing the spectrum of the singularity as a deeper invariant than the eigenvalues of the Algebraic Monodromy.
In particular, we show how an $f$-Jordan chain is obtained from several
$N$-Jordan chains by gluing them in the Brieskorn lattice.
\end{abstract} 

\section{Introduction}

Given a germ of a holomorphic function  $f:(\mathbb{C}^{n+1},0)\longrightarrow(\mathbb{C},0)$ with an isolated singularity one can associate to it the following objects:

1) The {\bf algebraic object,} consists of 
the Jacobian Algebra of  $f$
$$A_{f}:=\frac{\mathcal{O}_{\mathbb{C}^{n+1},0}}{J_{f}} $$
where
 $z_{0},\ldots,z_{n}$   are coordinates on $\mathbb{C}^{n+1}$,
$f_{j}=\partial f/\partial z_{j}$, $\mathcal{O}_{\mathbb{C}^{n+1},0}$  is the ring of germs of holomorphic functions at $0\in\mathbb{C}^{n+1}$ and 
 $J_{f}:=(f_{0},f_{1},\ldots,f_{n})\subset\mathcal{O}_{\mathbb{C}^{n+1},0}$ is the
   Jacobian ideal generated by the partial derivatives of $f$.  The dimension as a $\mathbb{C}$-vector space of $A_f$ is
 the   Milnor number $\mu:=dim_{\mathbb{C}}A_{f}<\infty$ (\cite{milnor}).
 $A_f$ carries naturally a non-degenerate bilinear form, the 
Grothendieck pairing (\cite{grifharis} p. 649) $$res_{f}:A_f \times A_f \to \mathbb{C},$$ 
$$
 res_f([h_{1}\, ],[h_{2}\, ])  = \left(\frac{1}{2 \pi i}\right)^{n+1}\int_\Gamma
\frac{h_1h_2}{f_0\cdots f_n}\,dz_0\wedge \cdots\wedge dz_n \hskip 5mm,\hskip 5mm \Gamma:=\{|f_0|=\cdots=|f_n|=\varepsilon\},\ \ 0<\varepsilon<<1.
$$
The nilpotent operator $M_f:A_f \lrar A_f$ given by multiplication with $f$  is $res_{f}$-symmetric. The index of nilpotency $m$ of $M_f$ is at most $n$ (i.e. $M_f^{m+1}=0 $ but 
$M_f^{m}\neq 0: A_f\rightarrow A_f$
 \cite{BrSkoda}). 
We can obtain from this data a flag of ideals in $A_f$:
$$A_f = W_{-m}(A_f,M_f)\supset W_{-m+1}(A_f,M_f)  \supset \cdots\supset    W_{m}(A_f,M_f) \supset  W_{m+1}(A_f,M_f) = 0 $$
 with extreme terms at both ends
$$ W_{-m+1}(A_f,M_f):=  Ann_{A_f}(f^m) \supset W_{m}(A_f,M_f) := (f^m) $$
and proceeding to construct the flag from induction
with
$$M_f^{m-1}:\frac{Ann_{A_f}(f^m)}{(f^m)}\longrightarrow
\frac{Ann_{A_f}(f^m)}{(f^m)}$$
which now has index of nilpotency $m-1$.

One then constructs another vector space of dimension $\mu$ which is obtained by ``chopping up" $A_f$ according to the filtration $W_*(A_f,M_f)$, and is called the graded module associated to the filtration:

$$Gr_W^*(A_f,M_f) := \oplus_{j=-m}^m Gr_W^j(A_f,M_f):=
\oplus_{j=-m}^m \frac{ W_{j}(A_f,M_f)}{ W_{j+1}(A_f,M_f)}.$$
Grothendieck residue will induce non-degenerate bilinear forms on $Gr_W^*(A_f,M_f)$ 
of type $res_f(M_f^j*,*)$ with nice orthogonality properties, and we will obtain that each of the graded pieces is made up of ``primitive" ones. It is in this graded module $Gr_W^*(A_f,M_f)$ that we can see the Lefschetz decomposition most clearly in this algebraic setting. Of course, it takes some reflection to understand this Lefschetz decomposition in the graded object with its bilinear form and its relation to the Jacobian Algebra $A_f$ and Grothendieck Residue pairing.

2) The {\bf topological object} consists of the $C^\infty$ locally trivial fibre bundle $\{X_t:=f^{-1}(t)\}_{t\in D^*}$ over a punctured disc $D^* \subset \mathbb{C}$ with fibre the ``Milnor Fibre", which is  a smooth $2n$ compact orientable differentiable manifold with boundary which is homotopically a bouquet of $\mu$ $n$-dimensional spheres (Milnor\cite{milnor}). 
The Geometric Monodromy map is the isotopy class
of the automorphism of the Milnor Fibre obtain
from lifting a loop around $0\in D^*$ to a diffeomorphism of the fibre of the locally trivial fibre bundle being the identity on the boundary of the fibre.
The $n^{th}$ dimensional cohomology groups with integer, rational or complex coefficients $\{H^n(X_t,K)\}_{t\in D^*}$, called the ``vanishing cohomology"  groups,  form a $K$-vector bundle over $D^*$ with a flat connection (the ``Gauss-Manin connection" (\cite{Bries}) whose Algebraic Monodromy map                                          
$
M_K:H^n(X_t,K) \longrightarrow H^n(X_t,K)  $
 is a fundamental linear invariant of the singularity. The Monodromy Theorem asserts that $M_K$ has roots of unity as eigenvalues, and that the size of the Jordan blocks of
 $M_K$ is bounded by $n+1$.
 For a sufficiently large $s$, $M_\ZZ^s$ will be a unipotent integer matrix and the $\mu\times\mu$ matrix
 
$$N:=\frac{1}{s}Log( M_\ZZ^s) =\frac{1}{s}
\sum_{j=1}^{n+1}\frac{1}{j} 
( M_\ZZ^s -Id_\mu)^j$$
 with rational coefficients satisfies that $e^N$ is the unipotent part of the factorization of $M=M_{s}M_u$ as a semisimple times a unipotent matrix.
 $N$ is a a nilpontent endomorphism of $H^n(X_t,\mathbb{C})$ that gives rise to a decomposition
 of the cohomology groups $H^n(X_t,\mathbb{C})$ `a la Lefschetz'.

  The bilinear form $Q$ in vanishing cohomology  comes from cup product on $V_t$. It is a $\pm$-symmetric bilinear form, and $N$ is antisymmetric with respect to it.
 All this allows us to give a decomposition of vanishing cohomology as made up of many pieces of basic primitive forms, seen as vector spaces with bilinear forms   $Q(N^j\bullet,J\bullet)$,  where $J$ is an involution whose eigenspaces are described
using the mixed Hodge structure of the singularity . Arithmetic appears here, since the bundle and the cup product are actually defined with integer coefficients  and the bilinear form has positive definite properties: the ``Riemann-Hodge bilinear relations" (\cite{grifharis}). All this gets codified into the ``Polarized Mixed Hodge Structure of Vanishing Cohomology".

   3) The   {\bf differential object} consists of the following: Denote as $H^n(X_\infty,\mathbb{C})$ the canonical cohomology fiber which is identified with the holomorphic sections associated to the vector bundle $$\underline{H}^{n}=\bigcup_{t\in D^{*}} H^{n}(X_{t},\mathbb{C}):=\{H^n(X_t,\mathbb{C})\}_{t\in D^*}.$$ Let $\{H_{e^{-2 \pi i \beta_j}}\}_{\beta_j\in\langle-1,0]\cap \QQ}$ be the set of generalized eigenspaces with respect to the semisimple part of the monodromy map $M_s$. Hence the elements on each $H_{e^{-2\pi i \beta_j}}$  become flat sections with the meromorphic Gauss-Manin connection. Consider the space
 $$\HH    :=\bigoplus_{k=0,\ldots,n}H  t^k, \hskip 1cm H:=  \bigoplus_{\beta_j \in \langle-1,0]\cap \QQ} H_{e^{-2 \pi i \beta_j}}t^{\beta_j-\frac{1}{2 \pi i}N}.$$
 The $V$-filtration in $\HH$ is defined by
 $$ V^\gamma := \bigoplus_{\gamma \leq \beta_j +k  } H_{e^{-2 \pi i \beta_j}}t^{\beta_j+k-\frac{1}{2 \pi i}N},$$
   the weight $W$-filtration
is formed by the topological $W_N$-filtration (that is, the weight filtration asociated to the nilpotent restriction map $N:H_{e^{-2\pi i \beta_j}}\rightarrow H_{e^{-2\pi i \beta_j}}$) in each summand, the $VW$ bi-filtration is a lexicographic 
combination of the two, with a double index, first the $V$-filtration followed by the $W$-filtration. Introduce the $\mathbb{C}$-linear map
$$\partial_t^{-1}:\HH \longrightarrow \HH,$$
which is defined on its summands by the relation (essentially, this map is the inverse of the local Gauss-Manin connection):
$$
\partial_t^{-1}(A_{\beta_j+k}t^{\beta_j+k-\frac{1}{2 \pi i}N}):=
 ({(\beta_j+k+1)I-\frac{1}{2 \pi i}N})^{-1}A_{\beta_j+k}t^{\beta_j+k+1-\frac{1}{2 \pi i}N})$$
 for $k=0,\ldots,n-1$ and $0$ for $k=n$.
The nilpotent map
is     $\partial_t^{-1}N$ and
 the bilinear form $<\ ,\ >_{\HH}$ is the
   orthogonal direct sum
 $$
 \HH = \bigoplus_{k=0,\ldots,(n-1)/2}^\perp[H  t^k\oplus H t^{n-k}]  \ \ \Bigl( \bigoplus^\perp H  t^{n/2}\ \  \hbox{for $n$ even} \Bigr) $$
 where in each summand the bilinear form has the form
 $$\langle vt^k,wt^{n-k} \rangle_{\HH}:=\displaystyle\frac{(-1)^{1+r}}{(2\pi i)^{n+1+r}} Q(\partial_t^{k}v t^{k}, (-1)^{n-k} \partial^{n-k}_{t}wt^{n-k}), \hskip 1cm v\in H\ ,\ w\in H\hskip 1cm r=-1  \hbox{ or }  0,$$
where $\partial_t^{k}v t^{k}$ and $ \partial^{n-k}_{t}wt^{n-k}$ are flat sections and $Q$ is the cup product as above.
The structure in $\HH$ then is topological (coming from ${H}$, $M$,  $N$ and $Q$)
and analytic due to $\partial_t$,
which can be interpreted as the Gauss-Manin connection on the sections of the bundle of $n^{th}$ primitive cohomology.

Consider now  the (Poincar\'e-Gelfand-Leray) residue (\cite{grifharis}). Namely, given a germ $\omega\in \Omega^{n+1}$ of an $(n+1)$-differential
form at $0\in \mathbb{C}^{n+1}$, its  residue $s(\omega)$ is a section of the cohomology bundle $\underline{H}$ over the punctured disk. Expand it in a
Laurent type expansion (\cite{malgrange}) and consider the first $(n+1)$-order terms in the expansion $s_n:\Omega^{n+1} \rightarrow \HH$, with image
$\HH^{\prime\prime}$. Introduce the subspace $\HH^\prime :=s_n(df \wedge \Omega^n)$ consisting of those expansions that can be realized by sections of the cohomology bundle $\underline{H}$ of the form $[\eta|_{X_t}]$, for $\eta \in \Omega^n$, where we are taking the de-Rham class in cohomology
$H^n(X_t,\mathbb{C})$ of the closed $n$-form $\eta|_{X_t}$ . The connection between the algebraic and the analytic then comes from the theorem
of Varchenko (\cite{varch1}) that asserts that the annihilator of the restriction of the bilinear form $<\ , \ >_{\HH}$ to $\HH^{\prime\prime}$
is $\HH^\prime$ and that the map $s_n$ induces  an isomorphism of the non-degenerate bilinear spaces
$$s_n:\Omega_f:= \frac{\Omega^{n+1}}{df\wedge\Omega^n} \longrightarrow \frac{\HH^{\prime\prime}}{\HH^\prime}$$
where we have put in the domain the bilinear form $res_f$ and in the image the induced non-degenerate bilinear form from $<\ , \ >_{\HH}$.

One induces the $V$-filtration on $ \HH^{\prime\prime}/\HH^\prime$ and then on $\Omega_f$ via the isomorphism $s_n$. Namely
$s_n(\omega)\in V^\gamma(\HH^{\prime\prime}/\HH^\prime)$ if there exists an $\eta \in \Omega^n$ with $s_n(\omega+df\wedge\eta)\in V^\gamma(\HH)$.
The principal term of the form $\omega \in \Omega^{n+1}$ is the $gr(V)$-smallest non-zero coefficient of $s_n(\omega)$.
The form $\omega  \in \Omega^{n+1}$ is original if the principal term of $s_n(\omega)$ cannot be $gr(VW)$-diminished  by adding an element  of the form $s_n(df\wedge \eta)$, for $\eta\in\Omega^n$.

Multiplication by $f$ in $\Omega^{n+1}$ satisfies
$s_n (f\omega)=ts_n (\omega)$, and hence  $s_n$ preserves the $V(\Omega_f)$-filtration inducing a degree 1 map in the associated graded module:
$$gr(f):Gr_V^{*}\Omega_f\longrightarrow Gr_V^{*+1} \Omega_f.$$

Another   theorem of Varchenko (\cite{varch3}) asserts that if we consider only the principal terms by considering the graded map $gr(s_n)$
when cutting with the $V$-filtration, then multiplication by $f$ corresponds to applying $\frac{-1}{2 \pi i }N$. This means that multiplication
by $f$ in the Jacobian module contains   more information than applying $\frac{-1}{2\pi i}N$ to its principal term, since it is acting on $\Omega_f$ and not only
on its principal term $Gr_V^{*}\Omega_f$.

This structure was clarified by M. Saito \cite{MSaito,MSaito2} and C. Hertling \cite{hertl3,hertl4} by choosing a
convenient basis  of $\Omega_f$ by original forms $\omega_1,\ldots,\omega_\mu \in \Omega^{n+1}$, which is adapted to multiplication by $f$. This basis  puts the bilinear
form $res_f$ in a canonical form and the principal part of  $s_n(\omega_j)$ is $A_jt^{\alpha_j-\frac{1}{2 \pi i}N}$ $\in  H_{e^{-2 \pi i \alpha_j}}t^{\alpha_j-\frac{1}{2 \pi i}N}$ which is adapted to multiplication by
$N$ and puts the bilinear form $S$ into canonical form (\cite[Prop. 5.1 and 5.4]{hertl3}). The rational numbers
$$-1 < \alpha_1 \leq \ldots\leq\alpha_\mu\leq n$$
are the spectral values. There is an increasing function $\nu_N:\{1,\ldots,\mu\} \longrightarrow \{1,\ldots,\mu+1\} $
and the $N$-adaptedness condition is $$  A_{\nu_N(j)}t^{\alpha_{\nu_N(j)}-\frac{1}{2 \pi i}N} =\partial_t^{-1} N(A_j)t^{\alpha_{ j }+1-\frac{1}{2 \pi i}N}\ \   \hbox{  with } A_{\mu+1}:=0.$$
The orbits of $\nu_N$ split the spectral values into $N$-chains of spectral values:
$$\{\alpha_1,\alpha_{\nu_N(1)}=\alpha_1+1,\alpha_{\nu_N^2(1)}=\alpha_1+2, \ldots,\alpha_{\nu_N^{\ell_1}(1)}=\alpha_1+\ell_1\},$$
$$\{ \alpha_2,\alpha_{\nu_N(2)}=\alpha_2+1,\alpha_{\nu_N^2(2)}=\alpha_2+2, \ldots\,\alpha_{\nu_N^{\ell_2}(2)}=\alpha_2+\ell_2\},\ldots$$

 The adaptedness condition for $f$ means that, beginning with the original 1-form $\omega_1$, its principal term is $A_1t^{\alpha_1-\frac{1}{2\pi i}N}$. Now $f\omega_1$ is not original, but its original representative is $f\omega_1-(\alpha_1+1)df\wedge \eta_1$, where $d\eta_1=\omega_1$ with principal term $[(\alpha_1+1)I-\frac{1}{2 \pi i}N]^{-1}NA_1$, etc.
  The binding of the spectral $N$-chains is produced from an argument as the following: taking the last basis element corresponding to the first $N$-chain, apply $f$ to it and substract $
 (\alpha_{\nu_N^{\ell_1}(1)}+1)df\wedge \eta_1$ with $\omega_{\nu_N^{\ell_1}(1)}=d\eta_{\ell_1}$:
 	$$f\omega_{\nu_N^{\ell_1}(1)}-(\alpha_{\nu_N^{\ell_1}(1)}+1)df\wedge \eta_{\ell_1}.$$
 	Since this element is in the end of an $N$-chain, the principal term that one would expect vanishes (i.e. $N^{\ell_{1}+1}A_1=0$), and so its
 	principal term corresponds to a spectrum point bigger than $\alpha_{\nu_N^{\ell_1}(1)}+1$. We continue the $N$-chain from this point as before. One has binded in this way the 2 $N$-chains together and one continues to obtain in this manner to build  the  $f$-chains.

   One now proceeds to analyse the bilinear forms $S$ and $res_f( \ \ ,\ \ )$ in Saito's basis. There is a function $\kappa:\{1,\ldots,\mu\}\longrightarrow\{1,\ldots,\mu\}$ which is $\kappa(j)=\mu+1-j$ except for some values where $\kappa(j)=j$, with $\alpha_{j}=(n-1)/2$. We have (\cite[Prop. 5.1 and Prop. 4.4]{hertl3})
$$S(L_jA_j,L_\ell A_\ell)=(-1)^{\ell+1+r}\delta_{\kappa(j),\ell}(2\pi i)^{n+1+r}\ \ ,\ \ \hskip 1cm r=-1 \hbox{ or } 0\; \ \ , \ \ res_{f}([\omega_j],[\omega_\ell])=\delta_{\kappa(j),\ell}.$$

We proceed to state the main results of this paper. Indeed, for a germ of a holomorphic function  $f:(\mathbb{C}^{n+1},0)\lrar(\mathbb{C},0)$ with an isolated critical point at $0\in \mathbb{C}^{n+1}$ we will show:
	\begin{itemize}
		\item[1)] Theorem \ref{thm:firstpartthm} which states that the bilinear forms $res_f(f^j\ ,\ )$
in the Jacobian module obtained by Grothendieck dulity and multiplication by $f^j$ in one factor can be expressed as a sum of the principal topological bilinear form $Q(N^j\ ,J\ )$ and
 weaker bilinear forms $Q_j$; being  $J$ an involution
 whose eigenspaces are expressed in terms of the Mixed Hodge structure of the Jacobian module.
 
 \item[2)] Corollary \ref{cor:mainthm} from which we obtain that the nature of the bilinear forms $Q_j$  arises from the bindings of the $N$-chains into the $f$-chains. 
\end{itemize}

The signatures of these bilinear forms enter as normalizing constants of a formula for computing indices of vector fields tangent to a hypersurface,
when working over the real numbers (\cite{Giraldo-GM}).

We  will use the simplified framework expounded in Hertling (\cite{hertl3}, \cite{hertl2})  of  work by Varchenko (\cite{varch1}, \cite{varch2}), K. and
M. Saito (\cite{Ksaito}, \cite{MSaito}), Scherk and Steenbrink (\cite{SherkSteen}). Recent related material are van Straten (\cite{duco}) and M. Saito (\cite{MSaito2}).

\section{Weight filtration on the Jacobian Algebra $A_f$}

\subsection{Choosing a basis for $A_f$ adapted to the Jordan block structure of $M_f$ }\label{sec:AfMf}

Let $f:(\mathbb{C}^{n+1},0)\longrightarrow(\mathbb{C},0)$ be a germ of a holomorphic function with an isolated singularity at $0$, $z_{0},\ldots,z_{n}$   coordinates on $\mathbb{C}^{n+1}$,
$f_{j}=\partial f/\partial z_{j}$ and $\mathcal{O}_{\mathbb{C}^{n+1},0}$  the ring of germs of holomorphic functions at $0\in\mathbb{C}^{n+1}$.
Denote the Jacobian ideal $J_{f}:=(f_{0},f_{1},\ldots,f_{n})\subset\mathcal{O}_{\mathbb{C}^{n+1},0}$ of $f$
  and
  the Jacobian Algebra
$$A_{f}:=\frac{\mathcal{O}_{\mathbb{C}^{n+1},0}}{J_{f}} $$
  with   Milnor number $\mu:=dim_{\mathbb{C}}A_{f}<\infty$.
Let
\begin{equation}
\label{eqn:multiplicationbyf}
M_f:A_{f}\lrar A_{f},\hskip .5cm[g ]\lrar[ f g ] 
\end{equation}
be the  operator defined by multiplication with $f$, where $[g]$ denotes the class of $g \in \mathcal{O}_{\mathbb{C}^{n+1},0}$ in $A_f$.
Multiplication by $f$ is  a symmetric operator with respect to the Algebra structure of $A_f$:
$$ [M_f(a)][b] = [a][M_f(b)]\hskip 1cm \Longleftrightarrow \hskip 1cm (fa)b=a(fb).$$
 $M_f$ is
nilpotent with index of nilpotency $m_{0}\leq n$ (i.e. $M_f^{m_0}\neq 0 \  ,\   M_f^{m_0+1}=0$, \cite{BrSkoda}).

We will choose an ordered basis (as a $\mathbb{C}$-vector space) of $A_f$ adapted to $M_f$: Choose $g_1 \in  A_f   -  Ann_{A_f }(f^{m_0}) $
 and the basis begins with 
 $$g_1,g_2:=fg_1,\ldots, g_{m_0+1}:=f^{m_0}g_1.$$
Imagine a
horizontal chain of length $2m_0$ with $m_0+1$ spheres  located at distance 2 in the chain
beginning with the extremes. Put on each of these spheres the element of the above partial basis  beginning on the left with 
$g_1,fg_1,f^2g_1,\ldots$
The mapping $M_f$ is a movement of the chain a step to the right of length 2 (each sphere moves to the next sphere to the right on the chain),
 and the extreme right sphere disappears under $M_f$
($f^{m_0+1}g_1=0$). We will call this a chain of size $m_0+1$, the number of basis vectors (or spheres) in the chain.
\begin{center}
    \centering 
    \includegraphics[width=.7\textwidth]{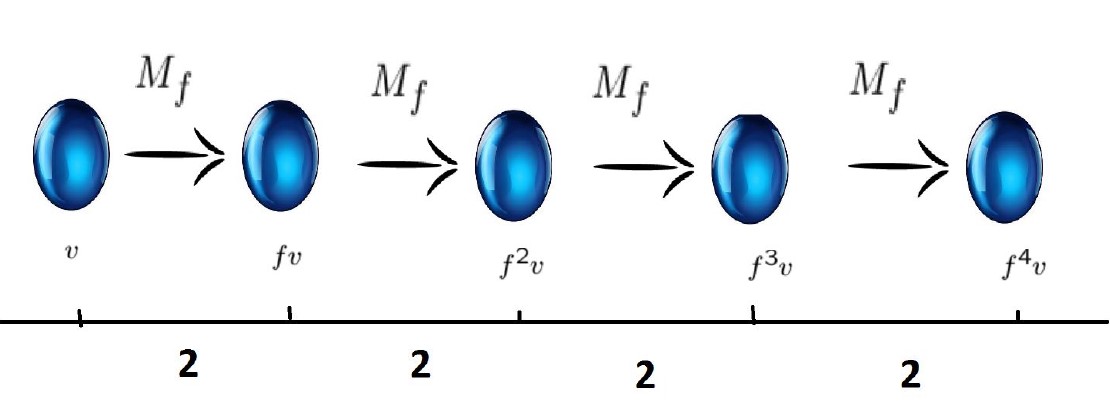}
\centerline{Fig. 1. Representation of the largest $M_f$-Jordan chain in $A_f$.}
    \label{fig1} 
\end{center}

Continue by chosing $g_{m_0+2} \in Ann_{A_f }(f^{m_1}) -  Ann_{A_f }(f^{m_1-1}) $,
with largest possible $m_1\leq m_0+1$ so that it is independent of    the already chosen basis elements $g_1,\ldots, g_{m_0+1}$,
so as to have the new partial basis $$g_1,\ldots,g_{m_0+1},g_{m_0+2},\ldots,g_{m_0+m_1+2}:=f^{m_1}g_{m_0+2},$$ etc.
Continue in this way we obtain a basis $\{g_k\}$, that we fix.
 Denote by $\ell_j$ the number of  chains    of size $j$ obtained. The numbers $(\ell_{m_0+1},\ldots,\ell_1)$  determine
  the Jordan-block type of $M_f:A_f \rightarrow A_f$, and for the Milnor number we have
$
 \label{eqn(a)}
 \mu = (m_0+1)\ell_{m_0+1} +\cdots +2\ell_2+\ell_1, 
 $
 since there are $\ell_1$ chains of size 1, $\ell_2$ of size 2, etc.
 
 Organize all the chains in such a way that the center of each chain is located over the point 0 and the
 negative numbers are to the left, as well as
 by size:  smaller   chains go on top. In this way we obtain a   pyramid like structure, which
 by definition is symmetric with respect to the involution on the vertical line over 0.
\begin{center}
    \centering 
    \includegraphics[width=.4\textwidth]{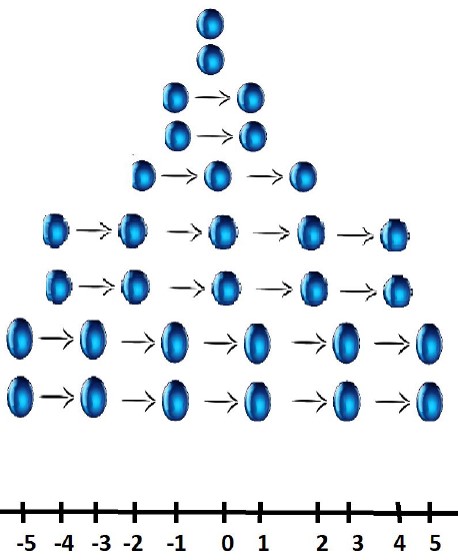}
    \centerline{Fig. 2.   Jordan block decomposition of $A_f$  of Jordan type $(2,2,0,1,2,2)$.}
    \label{fig2}
\end{center}

We define the $m$ platform of the pyramid as formed by those basis elements which are part of a Jordan chain
  of size $m$  (so $m$-spheres, $m-1$ arrows, length $2(m-1)$)
and the heigth of the $m$-platform is $\ell_m$, the number of Jordan chains of size $m$.
In this image we are developing  the $\ell_m$ Jordan blocks of size $m$ form the $m$-platform.
Imagine this platform as made up of $m-1$ blocks, each one of horizontal length 2 and
vertical height $\ell_m$ where the spheres constituting the platform are vertically evenly distributed
and horizontally are separated at distance 2.
The only different platform is the 1-platform, which looks more like an anthena, since its horizontal length is 0,
and it has vertical length $\ell_1$.
\begin{center}
    \centering 
    \includegraphics[width=.35\linewidth]{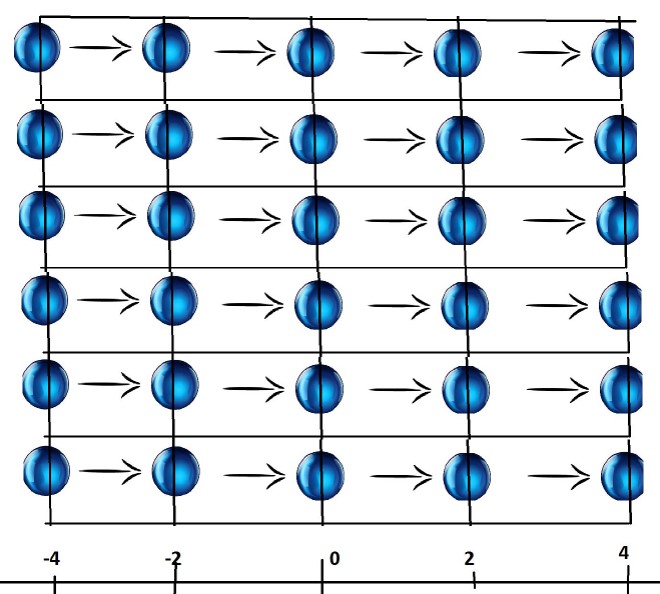}
   \centerline{Fig. 3. A 5-platform consisting of 4 blocks of height 6}
    \label{fig3} 
\end{center}

 The weight $w(g_j)$ of a basis element $g_j$ who is a member of a chain of size $m$ 
 and occupies the $k^{th}$ place in the chain, $k=1,\ldots,m$, is defined as $-m+2k-1$.
 It is the value of the projection of its sphere to the segment $[-m_0+1,m_0-1]$ at the base of the pyramid
 in the above construction. Its weight is $0$ if it is in the middle of a chain of size $m$ and it is $-m+1$ (or $m-1$) if it is on the
 extreme right (or left) of a chain of size $m$. The weight $w(g)$ of a general element   
$g=\sum_{i=1}^\mu a_ig_i \in A_f $, is
$$ w(g):=min\{w(g_i)\ / \ a_i\neq 0\}.$$

Since each block of the pyramid has length 2, and the platform on top has 1 chain less,
then this length 2 is divided into 2 segments of length 1 on each side of the platform.
So the blocks are not just `piled up', but they have a weaving structure:
\begin{center}
    \centering 
    \includegraphics[width=.4\textwidth]{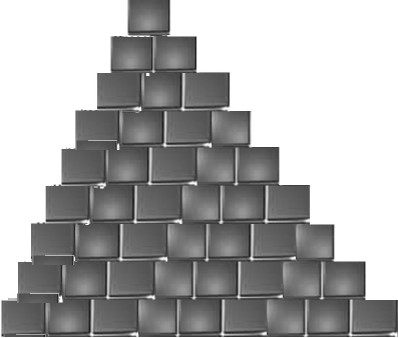}
     \centerline{Fig. 4.   Weaving Structure of the Pyramid of $(A_f,M_f)$.}
    \label{fig4}
\end{center}

 The $j^{th}$-column of the pyramid consists of those basis elements with weight $j$.
The height of the columns are, beginning on the left (or on the right):
$$\ell_{m_0}, \ell_{m_0-1},\ell_{m_0}+\ell_{m_0-2}, \ell_{m_0-1}+\ell_{m_0-3},\ldots,
\ell_{m_0} + \ell_{m_0-2} + \cdots +\ell_{m_0- 2k},\cdots
 ,\ell_{m_0 } + \ell_{m_0-2} , \ell_{m_0-1 }, \ell_{m_0 } $$
since the columns at distance 2 to the right and left of a given column consists of adding or deleting an additional piece of the column,
corresponding to the spheres on its uppermost platform,
as can be clearly seen by drawing the pyramid.
\begin{center}
    \centering 
    \includegraphics[width=.6\textwidth]{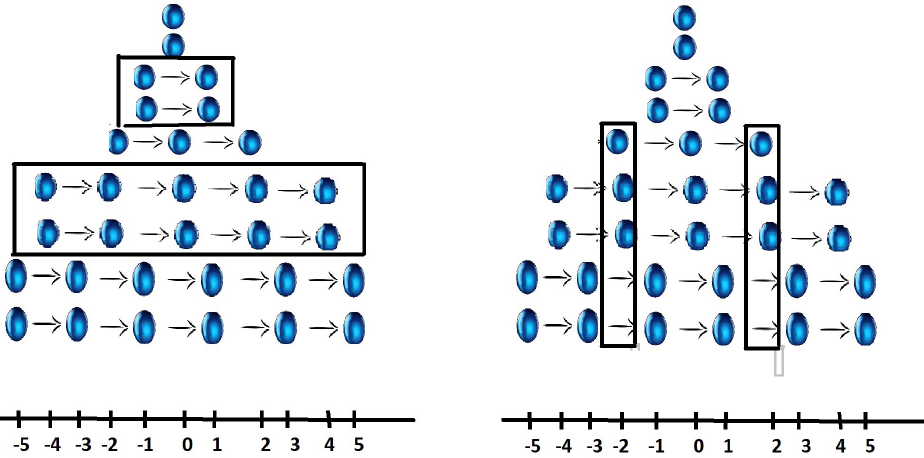}
 \centerline{Fig. 5.   Platform and Column Structure of the Pyramid of $(A_f,M_f)$.}
    \label{fig5}
\end{center}

 We reorder  the above basis $g_1,\ldots,g_\mu$ of $A_f$ to obtain a basis $\omega_1,\ldots,\omega_\mu$:
Begin the ordering with the basis elements on the columns   to the left, and on  each
column we begin the ordering with the elements in the blocks of the lower part of the column and moving up till the top block of that column,
before proceeding to the next column.
\begin{center}
    \centering 
    \includegraphics[width=.4\textwidth]{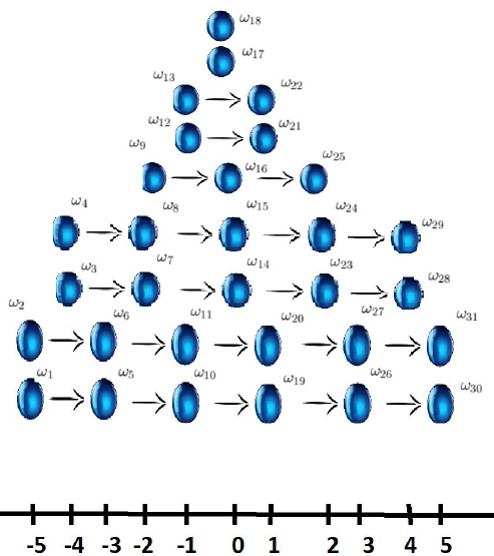}
\centerline{Fig. 6. Ordering the basis of $A_f$}
    \label{fig6}
\end{center}

We remember the collection of basis elements that appear in each       of the columns
\begin{align}\label{eqn:baseAf_omegas}
\{\omega_1,\ldots,\omega_\mu\}:=\{\{\omega_1,\ldots,\omega_{\ell_{m_0}}\}, \{  \omega_{\ell_{m_0}+1} ,\ldots,
\omega_{\ell_{m_0}+\ell_{m_1}}\},\\
 \{f\omega_1,\ldots,f^{}\omega_{\ell_{m_0}} ,   \omega_{2\ell_{m_0}+ \ell_{m_1}+1},\cdots,
\omega_{2\ell_{m_0}+\ell_{m_1}+\ell_{m_2}}\} , \ldots  \},\nonumber 
\end{align}
 forming $2m_0+1$ blocks,   that we number by the integer points in $[-m_0,m_0]$.
The above grouping of the basis will give decompositions of  endomorphisms $M_f$ of $A_f$ in the form of
block matrices.

Define the strictly increasing function 
\begin{equation}\label{eqn:aplicactionnu_omegas_base}
\nu_f:\{1,\ldots,\mu\} \rightarrow \{1,\ldots,\mu,\mu+1\}
\end{equation}

  by the condition $M_f(h_\ell) = \omega_{\nu(\ell)}$, with $\omega_{\mu+1}:=0$. This function
   codifies, in the above ordering of the basis, the effect of multipication by $f$: The matrix expression of $M_f$ in this basis is $(\delta_{j,\nu_f(j)})$.

\subsection{The Weight Filtration of the Jacobian Algebra}

This structure on the Jacobian Algebra appears due to the existence of a
canonical filtration $W(f)$ of $A_f$ induced by $M_f$, called the weight filtration (\cite{G-Sch}, \cite{schmid}):
\begin{equation}
\label{W}
A_f \supset
W_{-m_0}(M_f)\supset W_{ -m_0+1}(M_f)\supset \cdots \supset W_{0}(M_f)\supset \cdots\supset W_{m_0-1}(M_f)\supset W_{m_0}(M_f) \supset 0 .
\end{equation}

\begin{Def} The   element $W_j(M_f)$ of the weight filtration $W(M_f)$ of $A_f$ is defined as the vector space spanned by those basis elements
whose weight is greater than or equal to $j$, i.e. those which are
located to the right of $j$, in the above pyramid like structure of $A_f$.
\end{Def}

Note that the ordering  of the basis is constructed in such a way
that certain increasing segments of them $g_{k_1},\ldots,g_{k_2}$ generate
transversals to $W_j(M_f)$ in $W_{j-1}(M_f)$.

\begin{Lemma}\label{lemma:weigthbasisMf} 
	The weight fitration  is independent of the chosen basis.
\end{Lemma}
 \begin{proof} 
 We give an intrinsic definition of the weight filtration:
It proceeds by induction on the maximal length $m_0$ of the Jordan chains of $M_f$,
defining the extreme elements of the filtration  by $$W_{m_0}(M_f) := Ann(f^{m_0}) \ \ \hbox{and}\ \   W_{-m_0}(M_f):=(f^{m_0}).$$
For the induction step one considers the map induced by $M_f$ on $\frac{Ann(f^{m_0})}{(f^{m_0})} $.
If $M_f$ has type $(\ell_1,\ldots,\ell_r)$, then the induced map in $\frac{Ann(f^{m_0})}{(f^{m_0})} $ will
have type $(\ell_1,\ldots\ell_{r-2}+\ell_r,\ell_{r-1})$, since 
we have removed only the 2-extreme blocks of the $r$-platform,
so the remaining blocks will be of size $r-2$, and are incorporated to
 the $(r-2)$-platform of height $\ell_{r-2}$, 
to give the new platform of height $\ell_{r-2}+\ell_r$.
Now we only have Jordan chains of length strictly less than $m_0$, so induction hypothesis apply. We pull back to $Ann(f^{m_0})$
the obtained flag in $\frac{Ann(f^{m_0})}{(f^{m_0})} $, and we complete with the already chosen $W_{\pm m_0}(f)$. This completes the proof.
\end{proof}

\subsection{The $W(f)$-Graded Jacobian  $Gr^*_{W(f)}(A_f,M_f)$}
\begin{Def}
The $W(f)$-graded  Jacobian is a $\mathbb{C}$-vector space  of dimension $\mu$ defined by
 $$Gr^*_{W(f)}(A_f,M_f):=\bigoplus_{j={-m_0}}^{m_0}Gr_{W(f)}^j(A_f,M_f)
\hskip 2cm Gr_{W(f)}^j(A_f,M_f):=\frac{W_j(M_f)}{W_{j+1}(M_f)}.$$
\end{Def}

We have natural projections
$$gr_j:W_j(f) \longrightarrow \frac{W_j(f)}{W_{j+1}(f)} \subset Gr^*_{W(f)}(A_f,M_f)
$$
that may be thought as taking the `principal part' (with respect to the
weight filtration $W(f)$).
$Gr^*_{W(f)}(A_f,M_f)$ inherits a
graded basis $\{gr(\omega_1),\ldots,gr(\omega_\mu)\}$.
The map $ M_f $ sends  $W_j(M_f) $ to $ W_{j+2}(M_f)$, so there is an induced
degree $2$ map
$$gr(M_f): Gr^*_{W(f)}(A_f,M_f) \longrightarrow Gr^*_{W(f)}(A_f,M_f)
\hskip 1.5cm  gr(M_f)(gr(\omega_\ell)) = gr(f \omega_{\ell})=gr(\omega_{\nu(\ell)}).$$
It has in the graded basis the same matrix form as $M_f$:
$(\delta_{j,\nu(j)})$.

For $j=-m_0,\ldots,0$ define the primitive spaces
$$Prim^j(A_f,M_f) \subset Gr_{W(f)}^j(A_f,M_f)$$
as the vector spaces generated by $\{gr(\omega_j)\ / \ j \notin Im(\nu_f)\}$.
They correspond to the basis elements on the top block of the columns on the left side of the pyramid,
or the basis elements most to the left on each of the platforms.
They are the elements of the basis which are not divisible by $f$.
\begin{center}
    \centering 
    \includegraphics[width=.5\textwidth]{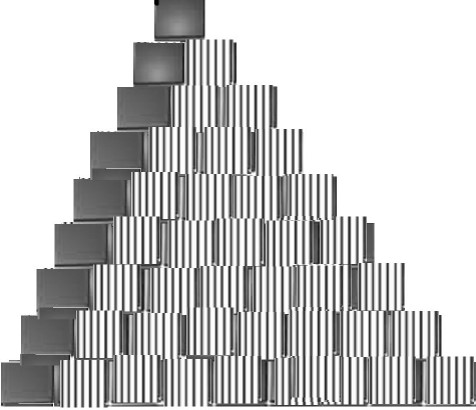}
\centerline{Fig. 7. The primitive subspaces    $A_f$ are colored.}
\end{center}

We obtain from this chain description  for $j=-m_0\ldots,m_0$ the direct sum
\begin{equation}
\label{prim2}
Gr_{W(M_f)}^j(A_f,M_f) = \bigoplus_{  - m_0\leq j-2k\ \leq 0 } M_f^k Prim^{j-2k}(A_f,M_f),
\end{equation}
which  describes the columns $Gr_{W(M_f)}^j(A_f,M_f)$ formed by piling up primitive blocks. The difference between one column and the
 column 2 steps to the left or right is the addition or deletion of one of the corresponding primitive block on top.
\begin{center}
    \centering 
    \includegraphics[width=.5\textwidth]{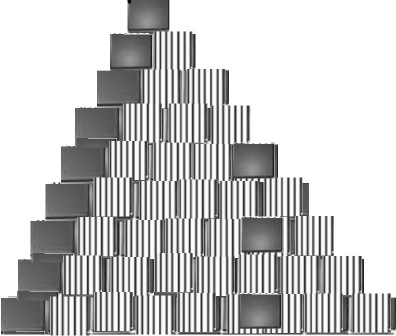}
\centerline{Fig. 8. The column description of $Gr_{W(M_f)}^j(A_f,M_f)$  by means of  primitive forms on the right.}
    \label{fig14}
\end{center}

 Intrinsically we have
$$Prim^{j}(A_f,M_f) := Ker [M_f^{-j+1}:Gr_{W(M_f)}^j(A_f,M_f))\longrightarrow Gr_{W(M_f)}^{-j+2}(A_f,M_f)]
\simeq \frac{Ann_{A_f}(f) \cap (f^{j})}{Ann_{A_f}(f) \cap (f^{j+1})},$$
since on a platform, the left block corresponds bijectively with the right block,
by applying $M_f$ an adequate number of times.

 Define the even and odd pieces of $Gr_{W(f)}^*(A_f,M_f))$ by
 $$Gr_{W(M_f)}^{even}(A_f,M_f)= \bigoplus_{k} Gr_{W(M_f)}^{2k}(A_f,M_f)
 \hskip 1cm, \hskip 1cm
 Gr_{W(M_f)}^{odd}(A_f,M_f) = \bigoplus_k Gr_{W(M_f)}^{2k+1}(A_f,M_f)$$
 and the induced maps
 $$
 gr(M_f): Gr_{W(M_f)}^{even}(A_f,M_f) \longrightarrow Gr_{W(M_f)}^{even}(A_f,M_f)\hskip 2mm, \hskip 2mm
  gr(M_f):Gr_{W(M_f)}^{odd}(A_f,M_f) \longrightarrow Gr_{W(M_f)}^{odd}(A_f,M_f)
 $$

We may summarize our conclusion in a way which is independent of the chosen basis as:

\begin{Prop} There is a filtration of $A_f$, called the $f$-weight filtration, canonically induced from the nilpotent map $M_f$ in the Jacobian Algebra $A_f$:

$$A_f = W_{-m_0}(A_f,M_f)\supset W_{-m_0+1}(A_f,M_f)  \supset \cdots\supset    W_{m_0}(A_f,M_f) \supset  W_{m_0+1}(A_f,M_f) = 0 $$
with the properties:

1) $M_f( W_{j}(A_f,M_f)) \subset W_{j+2}(A_f,M_f)$

2) Denote the associated graded objects by 
$$Gr_j(A_f,M_f):= \frac{W_{j}(A_f,M_f)}{ W_{j+1}(A_f,M_f)} \hskip 5mm , \hskip 5mm  Gr_{*}(A_f,M_f):=\bigoplus_{j=-m_0}^{m_0}  Gr_j(A_f,M_f)$$
$$Gr_{even}(A_f,M_f):=\bigoplus    Gr_{2j}(A_f,M_f) \hskip 5mm  , \hskip 5mm 
Gr_{odd}(A_f,M_f):=\bigoplus Gr_{2j+1}(A_f,M_f),$$
and the   $+2$-graded maps induced from $M_f$:
$$M_f: Gr_{even}(A_f,M_f)\longrightarrow Gr_{even}(A_f,M_f) \hskip 5mm,\hskip 5mm
M_f:Gr_{odd}(A_f,M_f)\longrightarrow Gr_{odd}(A_f,M_f)$$
with primitive pieces for $j=0,\ldots, m_0$:
$$Prim_{-j}(A_f,M_f):=Ker[M_f^{j+1}:Gr_{-j}(A_f,M_f)\longrightarrow  Gr_{j+2}(A_f,M_f)
$$
of dimension $p_j$ and
giving rise to a Lefschetz-type decomposition for $j=-m_0,\ldots,m_0$:
\begin{equation}
\label{primitive}
 Gr_{j}(A_f,M_f) = \bigoplus
M_f^k Prim_{j-2k}(A_f,M_f)
\end{equation}
and isomorphisms
$$
 M_f^j:Gr_{n-j}(A_f,M_f)\longrightarrow Gr_{n+j}(A_f,M_f).
$$

3) The interval $I:=\{1,\ldots,\mu\}$ may be divided into subintervals
$I=\{I_{-m_0},\ldots,I_{m_0}\}$
and each one further $I_j=\{I_{j,k}\}$
with  $I_{j,k}$ with $p_{j-2k}$ elements,
for $j=m_0,\ldots,0$, and we may choose a basis $g_1,\ldots,g_\mu$ of $A_f$  with an increasing map 
$\alpha:I\longrightarrow
\{I,\mu+1\} \ , \ g_{\mu+1}:=0$
such that:

a) $M_f(g_k)= g_{\alpha(k)}$,

b) $\{g_\ell\}$ for $\ell\in \{I_j,\ldots, I_{m_0}\}$ form a basis of $W_j(A_f,M_f)$,
 
 c) $gr(g_\ell)$ with $\ell\in I_{j,k}$ 
 are a basis of $M_f^k Prim_{j-2k}(A_f,M_f)$.

\end{Prop}

\subsection{  Grothendieck's Bilinear Form in the Jacobian Algebra }

 Define the linear transformation
 $$
L:\mathcal{O}_{\mathbb{C}^{n+1},0} \longrightarrow \mathbb{C} \ \ ,\ \ L(h):= (\frac{1}{2 \pi i})^{n+1}\int_\Gamma
\frac{h \underline{dz}}{f_0\cdots f_n}
$$ where  
$\underline{dz}=dz_0\wedge \ldots \wedge dz_n$ and $\Gamma$ is the $(n+1)$-real cycle $$
  \Gamma:=\bigl\{z\in\mathbb{C}^{n+1}:\, |f_{j}(z)|=\varepsilon\:,0\leq j\leq n\bigr\},\hskip 1cm d(\arg f_{0})\wedge\cdots\wedge d(\arg f_{n})> 0
  \hskip 5mm , \hskip 5mm \varepsilon < < 1. $$
On using Stokes' formula, one has $L(h)=0$ for  $h\in J_{f}$. So $L$ defines a $\mathbb{C}$-linear map $L: A_{f} \lrar \mathbb{C}$.

\begin{Def} Grothendieck's bilinear pairing is defined by
\begin{equation}
\label{ResPair}
 res_{f}:A_f\times A_f\lrar    \mathbb{C},\hskip .7cm res_f([h_{1}\, ],[h_{2}\, ])  = L([h_{1}h_{2}]) = \left(\frac{1}{2 \pi i}\right)^{n+1}\int_\Gamma
\frac{h_1h_2}{f_0\cdots f_n}\underline{dz}.
\end{equation}\end{Def}
It is  a nondegenerate pairing, by Grothendieck Local Duality (\cite{grifharis}, p. 659). The class  $Hess(f)$ 
of the Hessian determinant of $f$ generates the socle in $A_f$: the minimal
non-zero ideal in $A_f$ (\cite{EL}). The fundamental property of $L$ that is used to obtain a non-degenerate
bilinear form from the algebra structure of $A_f$ is that $L(Hess(f ))\neq 0$.
  
\section{The Bilinear Form in Cohomology for Hypersurfaces in Projective Space with an Isolated Singularity}

\subsection{Pencils of Hypersurfaces in Projective Space}

Let $f:\mathbb{C}^{n+1}\longrightarrow \mathbb{C}$ be a polynomial  such that $f^{-1}(0)$
extends   to the hyperplane at infinity,
$V_0 := \overline{f^{-1}(0)}$ as a smooth $n$ dimensional variety except at $0$, where it has an isolated singularity.
Then $V_t := \overline{f^{-1}(t)}$ is a projective manifold of dimension $n$ for $t\in \Delta-\{0\}$,
 a sufficiently small punctured disk.
The interesting part of the cohomology algebra $H^*(V_t,\QQ)$ of $V_t$ is in $H^n(V_t,\QQ)$. We have a non-degenerate
bilinear form given by cup product:

$$Q:H^n(V_t,\QQ) \times H^n(V_t,\QQ) \longrightarrow H^{2n}(V_t,\QQ)\simeq \QQ,$$
which is unimodular, symmetric for $n$ even and antisymmetric for $n$ odd.
Extending the coefficients to $\QQ \oplus i\QQ$ we obtain 2 types of extensions of the bilinear form $Q$, the $\mathbb{C}$-linear and the
Hermitian:
$$(u_0+iu_1)Q(v_0+iv_1), \hskip 3cm (u_0+iu_1)Q(v_0-iv_1)$$
 which receive block expressions, respectively:

\begin{equation}
\label{eqnQV1}
Q_\mathbb{C} :=
\begin{pmatrix}
Q & 0 \cr
0 & -Q
\end{pmatrix}
+ i \begin{pmatrix}
0& Q \cr
Q & 0
\end{pmatrix},
\hskip 2cm
Q_{\bar{\mathbb{C}}} :=
\begin{pmatrix}
Q & 0 \cr
0 & Q
\end{pmatrix}
+ i \begin{pmatrix}
0& -Q \cr
Q & 0
\end{pmatrix}.
\end{equation}
The first 3 matrices are symmetric if $Q$ is symmetric and the last one is symmetric if $Q$
is antisymmetric.
If $n$ is even, then $\QQ_{\bar{\mathbb{C}}}$ is Hermitian symmetric, and if $n$ is odd
$$i\QQ_{\bar{\mathbb{C}}} =
  \begin{pmatrix}
0&  Q \cr
-Q & 0
\end{pmatrix}
+ i \begin{pmatrix}
Q & 0 \cr
0 & Q
\end{pmatrix}
$$ is
Hermitian symmetric. If we extend the coefficients to $\mathbb{C}$ the matrix expressions of the extended
bilinear forms is the same (see \cite{hertl1}).

If we use the de Rham complex of $C^\infty$-differential forms on $V_t$ to
represent $H^n(V_t,\mathbb{C})$
, the bilinear form is obtained by cup product of two closed $n$-differential forms,
and integrating over $V_t$ the resulting $2n$-form.
The symmetry or anti-symmetry is then a consequence of the
alternating nature of the exterior algebra (see \cite{steenbrink0}).

\subsection{The Semi-simple Decomposition of $H^n(V_t,\mathbb{C})$}

 The  monodromy map
$M:H^n(V_t,\QQ) \longrightarrow H^n(V_t,\QQ)$ is the effect on cohomology of going around $t=0$ for the map $f$. It is a $Q$-automorphism: $Q(Mu,Mv) = Q(u,v)$ due to the functoriality
of the cup product.
Consider the map $M_s:H^n(V_t,\mathbb{C}) \longrightarrow H^n(V_t,\mathbb{C})$
defined by multiplication by $\lambda$ restricted to the generalized
$\lambda$-eigenspace $H^n(V_t,\mathbb{C})_\lambda $ of $M$, and $M_u:=M_s^{-1}M$ the unipotent part of $M$.
The monodromy automorphism $M$ is the product of its
semisimple and unipotent part $M=M_uM_s$.
Since $M$ is real we have
$${\overline{ H^n(V_t,\mathbb{C})_\lambda}} =  H^n(V_t,\mathbb{C})_{\overline{\lambda}} $$
so that if
$$H^n(V_t,\RR)_{\lambda,{\overline{\lambda}}} = [H^n(V_t,\mathbb{C})_ \lambda  \bigoplus H^n(V_t,\mathbb{C})_{\overline{\lambda}}]\bigcap H^n(V_t,\RR)$$
we have a $Q$-orthogonal   direct sum, $M$-invariant decomposition

\begin{equation*}\label{eqnob} 
H^n(V_t,\RR) = H^n(V_t,\RR)_1 \bigoplus H^n(V_t,\RR)_{-1} \bigoplus [\bigoplus_{Im \lambda>0} H^n(V_t,\RR)_{\lambda,{\overline{\lambda}}}]
\end{equation*}

\begin{equation}\label{eqnoc}
H^n(V_t,\QQ) = H^n(V_t,\QQ)_1 \bigoplus H^n(V_t,\QQ)_{ 1}^{\perp_Q}, 
\end{equation} 
$$ H^n(V_t,\QQ)_{ 1}^{\perp_Q}:=H^n(V_t,\QQ)_{-1} \bigoplus [\bigoplus_{Im \lambda>0} H^n(V_t,\QQ)_{\lambda,{\overline{\lambda}}}]$$

\subsection{The Unipotent Decomposition of $H^n(V_t,\QQ)$}\label{sec:Nu_N}

  Define
$$N_V:=  log(M_u):=
 \sum_{j\geq 1} \frac{(-1)^{j+1} (M_u-Id)^j}{j}.$$ deleting the subscript $V$, we have which is $Q$-antisymmetric: $Q(N\bullet,\bullet)=-Q(\bullet,N\bullet)$
and nilpotent, say $N^{r_0+1}=0$. Since $M_sM_u=M_uM_s$, the description that follows may be applied to each summand in \eqref{eqnoc} and
then take the direct summand, without being explicit about this.

There is a
canonical filtration $W(N)$ of $H^n(V_t,\QQ)$ induced by $N$,   the weight filtration (\cite{schmid})
\begin{equation}\label{eqn:weightFiltrationN_V}
0 \subset W_{-r_0}(N)\subset W_{-r_0+1}(N)\subset \cdots \subset W_{0}(N)\subset \cdots\subset W_{ r_0-1}(N) \subset W_{r_0}(N)\subset H^n(V,\QQ) .
\end{equation}
We may visualize the weight filtration by choosing a Jordan basis of $H^n(V,\QQ)$ with respect to $N$: To each $\ell$-Jordan block associate a
horizontal chain of     length 2$\ell$, marking the even integer points on the chain and putting on each
point an element of the basis  beginning on the left $A,NA,\ldots,N^{\ell -1}A$.
The mapping $N$ is a movement of the points on  the chain 2 steps to the right, and the extreme right points disappear under $N$.
 Organize all the chains in such a way that the center of each chain is located over the point 0.
The element $W_j(N)$ of the weight filtration are then spanned by those elements of the basis which are
located to the right of $j$.

Order the above basis of $H^n(V_t,\QQ)$ beginning with a basis for $W_{-r_0}$, then completing it to a basis of $W_{-r_0+1}$, etc.:
 $\{\{A_1,\ldots,A_{\ell_1}\}, \ldots,\{\ldots,A_\ell\}\}$. We remember the collection of basis elements that appear in each step, so as to obtain representations of  endomorphisms of $H^n(V_t,\QQ)$ in the form of
block matrices.
There is a decreasing function $$\nu_N:\{1,\ldots,\ell\} \rightarrow \{0,1,\ldots,\ell\}$$
with the property that $N(A_k) = A_{\nu_N(k)}$, with $A_0:=0$. The matrix of $N$ in this basis is $(\delta_{j,\nu_N(j)})$.

Introduce the $W(N)$-graded  Cohomology Algebra
 $$Gr_{W(N)}(H^n(V_t,\QQ),N):=\bigoplus_{j={-r_0}}^{r_0}Gr_{W(N)}^j(H^n(V_t,\QQ),N),
\hskip 1.5cm Gr_{W(N)}^j(H^n(V_t,\QQ),N):=\frac{W_j(N)}{W_{j-1}(N)}$$
 with graded basis $\{gr(A_1),\ldots,gr(A_\ell)\}$.
The induced map
$$gr(N): Gr^{N}(H^n(V_t,\QQ),N) \longrightarrow Gr^{N}(H^n(V_t,\QQ),N),
\hskip 1.5cm  gr(N)(gr(A_j)) = gr(A_{\nu_N(j)})$$
has in the graded basis the same matrix form as $N$:
$(\delta_{j,\nu(j)})$.

For $j=0,\ldots,m_0$ the primitive spaces
$$Prim_j(H^n(V_t,\QQ),N) \subset Gr_{W(N)}^j(H^n(V_t,\QQ),N)$$
are generated by $\{gr(A_j)\ / \ j \notin Im(\nu_N)\}$ and
they are  the beginning of the chains
description  of $Gr_{W(N)}(H^n(V_t,\QQ),N)$.
We obtain from the chain image for $j=-m_0\ldots,m_0$ the direct sum
\begin{equation}
\label{prim1}
Gr_{W(N)}^j(H^n(V_t,\QQ),N) = \bigoplus_{0\leq j+2k\leq m_0} N^k Prim_{j+2k}(H^n(V_t,\QQ),N),
\end{equation}
which  describes the columns $Gr_{W(N)}^j(H^n(V_t,\QQ),N)$ formed by piling up primitive blocks. The difference between one column and the
 column 2 steps to the left or right is the addition or deletion of one of the corresponding primitive block on top. Intrinsically we have
$$Prim_{j}(V) := Ker [N ^{j+1}:Gr_{W(N)}^j(H^n(V_t,\QQ),N))\longrightarrow Gr_{W(N)}^{-j-2}(H^n(V_t,\QQ),N)]
\simeq Ann_{H^n(V_t,\QQ)}(N) \cap (N^{j+1}).$$

\subsection {Vanishing Cohomology and its Non-degenerate Pairing}\label{sec:4.1}

Let $f:\mathbb{C}^{n+1} \longrightarrow \mathbb{C}$ be a germ of a holomorphic function at $0$ with an isolated singularity. Following Brieskorn \cite{Bries} and Scherk \cite{Sherk} (cf. \cite{SherkSteen}) one can make an analytic change
of coordinates in such a way that $f$ is a polynomial map of sufficiently large degree and such that $V_t:=\overline{f^{-1}(t)}$ its closure in $\mathbb{C} P^{n+1}$ is  a smooth
hypersurface at infinity, and such that denoting $X_t:=V_t \cap B$ with $B$ a sufficiently small ball in $\mathbb{C}^{n+1}$ the restriction map gives
rise to the exact sequence

\begin{equation}
 0 \longrightarrow Ker(M-Id) \longrightarrow H^n(V_t,\QQ) \longrightarrow  H^n(X_t,\QQ)\longrightarrow 0
\end{equation}
Since $Ker(M-Id)=Ker(N)$ (\cite{Sherk}) we have an isomorphism
$$\phi: H^n(X_t,\QQ) \longrightarrow  \frac{H^n(V_t,\QQ)}{Ker(N)}$$
and we  define the Hertling-Steenbrink polarization (see \cite{hertl1}, \cite{hertl2} and \cite{hertl3}) bilinear form  \begin{equation}
Q_{X_t}:H^n(X_t,\QQ) \times H^n(X_t,\QQ) \longrightarrow \QQ \hskip 1cm
Q_{X_t}(\bullet,\bullet):=Q(2\pi i N_V \phi \bullet, \phi \bullet)
\end{equation}
 as in \eqref{eqnQV1}.

  \section{ The Differential Description of Grothendieck duality}

  The Brieskorn lattices $H_0^{\prime\prime}$
  and $H_0^\prime$
are defined by the exact sequence of $\mathbb{C}$-vector spaces:
$$0  \longrightarrow H_0^\prime:=\frac{df \wedge \Omega^n_{\mathbb{C}^{n+1},0}}{df \wedge d\Omega^{n-1}_{\mathbb{C}^{n+1},0}}\longrightarrow H_0^{\prime\prime}:=\frac{ \Omega^{n+1}_{\mathbb{C}^{n+1},0}}{df \wedge d\Omega^{n-1}_{\mathbb{C}^{n+1},0}}\longrightarrow \Omega_f:= \frac{ \Omega^{n+1}_{\mathbb{C}^{n+1},0}}{df \wedge \Omega^n_{\mathbb{C}^{n+1},0}} \longrightarrow 0.$$

The cohomology bundle over a punctured disk  $R^nf_*\mathbb{C}_{\mathbb{C}^{n+1},0}$ will be denoted by $\underline{H}^n$. It is a $\mu$-dimensional flat
bundle with the Gauss-Manin connection $\partial_t$ whose fiber over $t$ is the
vanishing cohomology group $H^n(X_t,\mathbb{C})$. 

Taking the universal covering $e:D_{\infty}\lrar D^{*}$, $\xi\;\mapsto \exp{2\pi i\xi}$  of $D^{*}$ and the inclusion map $i:D^{*}\rightarrow D$ we have the canonical Milnor fibre given by the pullback
\[\xymatrix @C=.75pc {
	X_{\infty}:=X^{*}\times_{\Delta{*}}\Delta_{\infty}\ar[d]\ar[rr]&&X^{*}\ar[d]^-{f}\\
	D_{\infty}\ar[rr]^-{e}&&D^{*}
	,}
\]
where $f:X^{*}\lrar\Delta^{*}$ is a  $\mathcal{C}^{\infty}$ fiber bundle whose fibres $X_{t}:=f^{-1}(t)\cap X$ (see \cite{hertl2,hertl3}).

So we have homotopy equivalences given by the inclusions
\begin{equation*}\label{FIBRAMILNORhomotopicos}
\xymatrix{X_{u(\tau)}\simeq (X_{\infty})_{\tau}\ar @{^{(}->}[r]&X_{\infty}, \hskip 1cm x\mapsto \varsigma_{\tau}(x)=(x,\tau),
}
\end{equation*}

and so isomorphisms
\begin{equation}\label{eqn:isosCanonicalfibers}
 \xymatrix@C=3pc{
	H^{n}(X_{\infty},\mathbb{C})\ar[r]^-{\varsigma^{*}_{\tau}}& H^{n}(X_{t},\mathbb{C}),\hskip 1cm H_{n}(X_{t},\mathbb{C})\ar[r]^-{(\varsigma_{\tau})_{*}}& H_{n}(X_{\infty},\mathbb{C}),}
\end{equation}
where $t=e(\tau)$ (see \cite{kulikov}).

The bundle $\underline{H}^n$ has a natural  $\partial_t$-invariant non-degenerate bilinear form
obtained by gluing the bilinear form explained in section \ref{sec:4.1} for each  $X_t$. Such bilinear form induces, up to conjugation by the isomorphisms in \eqref{eqn:isosCanonicalfibers}, an equivalent bilinear form defined on $H^n(X_\infty,\mathbb{C})$ which without loss of generality we also denote as $Q$. On the other hand, we have monodromy maps $M_t$ on $H^{n}(X_{t},\mathbb{C})$ which induce the {monodromy} map $M$ on $H^{n}(X_{\infty},\mathbb{C})$: 
$$
\xymatrix  @C=2pc  @R=2pc{H^{n}(X_{\infty},\mathbb{C})\ar[r]^-{M}\ar[d]_-{\zeta^{*}_{\tau}}& H^{n}(X_{\infty},\mathbb{C})\ar[d]^-{\zeta^{*}_{{\mathbb{C}\tau}}}\\
	H^{n}(X_{t},\mathbb{C})\ar[r]^-{M_t}& H^{n}(X_{t},\mathbb{C}).
}
$$

With respect to the decomposition $M = M_{s}M_{u}=M_{u}M_{s}$  into semisimple and unipotent parts. there is a eigenspace decomposition
\begin{equation*}
	H:=H^{n}(X_{\infty},\mathbb{C})=\bigoplus_{\lambda} H_{\lambda},
\end{equation*}
with respect to $M_{s}$, i.e.,   $H_{\lambda}:= \ker(M_{s}-\lambda I)$. Set $ H_{\neq 1}:=\bigoplus_{\lambda} H_{\lambda\neq1}$ and let $N:=-\frac{1}{2\pi i}\log M_{u}\in End_{\mathbb{C}}\bigl(H\bigr)$ be the logarithm of the unipotent part of the monodromy which is nilpotent (by monodromy theorem) with  nilpotence index $r_0\leq n+1$ (see \cite{hertl2} and \cite{kulikov}.) Hence, we have a similar description as shown in subsection \ref{sec:Nu_N}. In particular, we have an application $\nu_N$  as in \eqref{eqn:weightFiltrationN_V} that encodes the Jordan blocks of $N$ in terms of a basis that describes the corresponding weight filtration which we also denotes as $W_{\bullet}(N)$.
 
A class of holomorphic (univalued) sections of   $\underline{H}^n$ may be represented
by Laurent-type series expansions
$$V^{>-\infty}:=\Biggl\{\sum\limits_{\substack{j=1\\ k\in \ZZ}}^{r}t^{[(\beta_j+k)I-\frac{1}{2 \pi i}N]}A_{j,k}(t)\Biggr\},$$
where $A_{j,k}(t)$ is a Gauss-Manin (multivalued) flat section which takes values in $\underline{H}^n_{e^{-2 \pi i \beta_j}},\ \ \beta_j\in(-1,0]$.
The convergence of the corresponding series holds in each sector $a<\arg t<b$ with  $|t|$   small.
The
$V$-filtration is defined by
$$V^{\beta (>\beta)}:=\Biggl\{\sum\limits_{\substack{j=1\\ k\in \ZZ}}^{r}t^{[(\beta_j+k)I-\frac{1}{2 \pi i}N]}A_{j,k}(t),\hskip .5 cm \beta_j+k\geq (> )\beta\Biggr\}$$

$$ V^{-\infty} :=\bigcup_{\beta\to -\infty}V^{\beta}\hskip 1cm,\hskip 1cm \hskip 1cm \cdots V^{r-1} \supset V^r \supset V^{r+1}\cdots$$
and we denote its graded pieces by
$$\xymatrix@C=4pc{C_{\beta_j+k}:= Gr_V^{\beta_j+k}(V^{>-\infty}) = \frac{V^{\beta_j+k}}{V^{> \beta_j+k}}&\ar[l]^-{\psi_{\beta_j+k}}_-{\simeq}H^n_{e^{-2 \pi i\beta_j }}}, \ \ \hbox{where}\ \ \psi_{\beta_j+k}(A_{j,k}):=t^{[(\beta_j+k)I-\frac{1}{2 \pi i}N]}A_{j,k}(t) .$$
The $V$-weight $v(A)$ of $A \in V^{-\infty}$ is the  smallest $\beta_j+k$ such that $A \in V^{ \beta_j+k}$, and so, it is
the smallest rational number where $A_{\beta_j,k}\neq0$ in the expansion of $A$.

The expression of the Gauss-Manin connection   is
$\partial_t:V^{-\infty}\longrightarrow V^{-\infty}$ that has the direct sum expression
 $$
 \partial_t = \psi_{\beta_j+k-1} \bigl((\beta_j+k)I-\frac{1}{2\pi i} N\bigr)\psi^{-1}_{\beta_j+k} :C_{\beta_j+k} \longrightarrow C_{\beta_j+k-1}
 $$
on using Leibniz rule in the computation:
 $$\partial_tt^{[(\beta_j+k)I-\frac{1}{2 \pi i}N]}A_{j,k}(t) :=
 \partial_te^{([(\beta_j+k)I-\frac{1}{2 \pi i}N]\log(t))}A_{j,k}(t)=
  t^{[(\beta_j+k)I-\frac{1}{2 \pi i}N]}\biggl(\frac{(\beta_j+k)I-\frac{1}{2 \pi i}N}{t}\biggr)A_{j,k}(t)
, $$
and the maps $\partial_t^k$ correspond to the linear maps $L_k:C_{\beta_j+k}\longrightarrow C_{\beta_j}$ up to conjugation by $\psi$:
\begin{equation}
\label{eqn:dt}
\partial^{k}_t=\psi_{\beta_j} [(\beta_j+1)I-\frac{1}{2 \pi i}N]\cdots [(\beta_j+k)I-\frac{1}{2 \pi i}N]\psi^{-1}_{\beta_j+k},
\end{equation}
that is, $$L_k=[(\beta_j+1)I-\frac{1}{2 \pi i}N]\cdots [(\beta_j+k)I-\frac{1}{2 \pi i}N].$$
Introduce the spaces
$$\HH_0:=\frac{V^{>-1}}{V^{0}} \quad,\quad \HH_k:=\frac{V^{k-1}}{V^{k}}\ , \ k=1,\ldots,n \hskip 1cm,\hskip1cm \HH = \bigoplus_{k=0}^n \HH_k=\frac{V^{>-1}}{V^n}.$$
$\HH$ consists of those coefficients in the finite expansion between $\langle-1,n\rangle$,
with its induced $V$-filtration. We obtain induced $\mathbb{C}$-linear maps
$$\partial_t^j:\frac{V^{-1+j}}{V^{n}} \subset \HH \longrightarrow \HH \hskip 1cm,\hskip1cm
\partial_t^{-1}:\HH \longrightarrow \HH$$
by applying the maps in \eqref{eqn:dt}. Each $\HH_k,\,k=1,\ldots,n$, is isomorphic to vanishing cohomology $H^n(X_\infty,\mathbb{C})$:
$$\xymatrix{H^n(X_\infty,\mathbb{C})&\ar[l]_-{\psi^{-1}}^-{\simeq}\HH_0\oplus C_0&\ar[l]_-{\partial^{k}_t}^-{\simeq}\HH_{k}}\,,$$
where $\psi:=\bigoplus_{-1<\beta_{j}\leq 0}\psi_{\beta_j}$, and $\HH_0$ is isomorphic to $H^{n}_{\neq 1}:=\bigoplus_{-1<\alpha< 0}H^{n}_{e^{-2\pi i \alpha}}$.

Introduce in $\HH$ the non-degenerate bilinear form $\langle\ ,\ \rangle_\HH$  as the  orthogonal
decomposition
$$\HH=[(\HH_0\oplus C_0) \oplus \HH_{n}]\bigoplus^\perp_{\ell=1,.,\frac{n-1}{2} }[\HH_\ell \oplus \HH_{n-\ell}] \bigoplus^\perp \HH_{n/2}$$
and in each factor it is defined as
$$\langle v_1t^\ell,v_2t^{n-\ell}\rangle_\HH =
\displaystyle\frac{(-1)^{1+\lfloor\beta_j\rfloor}}{(2\pi i)^{n+1+\lfloor\beta_j\rfloor}} Q(L_kV_1,(-1)^{n-\ell}L_{n-k}V_2) ,\ \
v_1t^{\ell}\in Gr_V^{\beta_j+\ell},\ \ v_2t^{n-\ell} \in Gr_V^{-\beta_j-1+n-\ell},
$$
where $$L_\ell V_1\in H^{n}_{e^{-2\pi i \beta_{j}}},L_{n-\ell}V_2\in H^{n}_{e^{2\pi i \beta_{j}+1}} $$
are such that
$$
\psi_{\beta_j}^{-1}\partial^{k}_tv_1t^{\ell}= L_{\ell}\psi^{-1}_{\beta_j+k}(v_1 t^{\ell})=L_\ell V_1,$$
$$
\psi_{-\beta_j-1}^{-1}\partial^{n-\ell}_{t}v_1t^{n-\ell}= L_{n-\ell}\psi^{-1}_{\beta_j+k}(v_2t^{n-\ell})=L_{n-\ell} V_2,$$
and 
$\lfloor\beta_j\rfloor=-1,0,$ and $0$ otherwise. $Q$ is the cup product in flat sections.

The Gelfand-Leray residue defines a map
$$\Omega^{n+1}_{\mathbb{C}^{n+1},0} \longrightarrow \underline{H}^n\hskip 1cm,\hskip 1cm s(\omega)(t):=res_{X_t}\biggl[\frac{\omega}{f-t}\biggr]\in H^n(X_t,\mathbb{C})   $$
and it induces the  period map
$$s:H^{\prime\prime}_0 \longrightarrow V^{-1}$$
whose restriction to $H^\prime_0$ has the expression
$$s(df\wedge\eta) = [\eta|_{X_t}] \in H^n(X_t,\mathbb{C}), \hskip 1cm \eta \in \Omega^n_{\mathbb{C}^{n+1},0}.$$
The period map $s$  is injective and it satisfies $s(H_0^{\prime\prime})\supset s(H_0^{\prime})\supset V^{n}$. We identify the Brieskorn lattice $H_0^{\prime\prime}$
with its image in $V^{>-1}$. The   map
$\partial_t^{-1}: V^{>-1} \longrightarrow V^{>0}$ is bijective and it defines an injective map
$$\partial_t^{-1}:H_0^{\prime\prime}
\longrightarrow H_0^{\prime\prime}  \hskip .5cm, \hskip .5cm   s(\omega) := s(d\eta) \mapsto s(df\wedge \eta) \hskip .5cm, \hskip .5cm \partial_t^{-1}H_0^{\prime\prime} =H_0^{\prime}, $$
hence providing an isomorphism
\begin{equation}\label{eqb:mapsOmegaToBrieskornLatt}
s:\Omega_f \longrightarrow \frac{H_0^{\prime\prime}}{\partial_t^{-1}H_0^{\prime\prime} }.
\end{equation}

Introduce the subspaces
$$\HH_0^{\prime } := \frac{ H_0^{\prime } }{ H_0^{\prime }  \cap V^{n}}
\subset
\HH_0^{\prime\prime} := \frac{ H_0^{\prime\prime} }{ H_0^{\prime\prime}  \cap V^{n}}
\subset \HH \hskip   1cm, \hskip 1cm
\Omega_f \stackrel{\simeq}{\rightarrow} \frac{\HH_0^{\prime\prime}}{\HH_0^{\prime}}
$$
whose elements consist of the coefficients in the expansion between $\langle-1,n\rangle$
which are realized   by  differential forms in $df\wedge\Omega^n$ or in $ \Omega^{n+1}$, respectively. The las isomorphism is induced by $s$ in \eqref{eqb:mapsOmegaToBrieskornLatt}.

The next theorem is essentially due to Varchenko \cite{varch1} and its variant that we present here is following Hertling \cite[Proposition 4.4]{hertl3}:

\begin{Teo}[Grothendieck duality]
	\label{thm:varchenko}
The radical of the restriction of the bilinear form $<\ ,\ >_\HH$ to $\HH_0^{\prime\prime}$ is $\HH_0^{ \prime}$ , and the
induced non-degenerate bilinear form in $\frac{\HH_0^{\prime\prime}}{\HH_0^{\prime}}$ via the identification with $\Omega_f$ is
Grothendieck residue $res_f$.
\end{Teo}

We may induce the  descending filtration  $V(V^{-1})$ to a  filtration $V(\HH_0^{\prime\prime})$ and then to a
filtration   $V(\Omega_f)$ in the Jacobian module. Explicitly: $[\omega]\in V^\beta(\Omega_f)$  if $\beta$ is the smallest rational number such that there is $\eta \in \Omega^n_{\mathbb{C}^{n+1},0}$ such
that $v_0(s(\omega + df\wedge \eta))= \beta$.

The spectrum $\{\alpha_1\leq\ldots \leq\alpha_\mu \}$ is formed by those rational numbers in $\langle-1,n\rangle\cap \QQ$ where $Gr_V^{\alpha_j}(\Omega_f)\neq0$. The spectrum can be interpreted as choosing logarithms of the eigenvalues
of the monodromy $M$: $\{e^{-2 \pi i \alpha_1},\ldots,e^{-2 \pi i \alpha_\mu}\}=\{e^{-2 \pi i \beta_1},\ldots,e^{-2 \pi i \beta_r}\}$,
and so they are finer invariants than the eigenvalues of the monodromy.
The choosing of the logarithms is  unveiled by the differential description.

\section{Normal form of the bilinear forms $res_f(f^j\ , \ )$ in Saito--Hertling basis}

In this section we will give a normal form for the higher bilinear forms $res_{f,0}(f^j\ ,\ ),\ 1\leq j\leq \ n+1$ in terms of the Saito-Hertling basis.

First, following  Hertling \cite{hertl3} (cf. \cite{MSaito}), one introduces the Hodge filtration in vanishing cohomology:
 $$H^n(X_\infty,\mathbb{C}) \supset F^0 \supset F^1 \supset \ldots \supset F^n \supset 0.
 $$
 It is compatible with the semi-simple decomposition in vanishing cohomology,
and it is defined for $p=0,\ldots,n$ by
$$F^{p}(H^n_{e^{-2 \pi i \alpha_{j}}})=\psi^{-1}\biggl(\frac{(V^{\alpha_j}\cap \partial_t^{n-p}H_0^{\prime\prime})+V^{>\alpha_j}}{V^{>\alpha_j}}\biggr).$$

\begin{Teo}[C. Hertling \cite{hertl2,hertl3}, cf. \cite{steenbrink0}] Let $f:(\mathbb{C}^{n+1},0)\rightarrow (\mathbb{C},0)$ be a holomorphic germ with an isolated singularity. Then, the vanishing
cohomology $H^n(X_\infty,\QQ)$ has a polarized mixed Hodge Structure (PMHS): $(F^*,W_*,S)$. This means that there is a PMHS of weight $n$ on $H^n(X_\infty,\QQ)_{\neq 1}$ and a PMHS of weight $n+1$ on $H^n(X_\infty,\QQ)_{1}$.
\end{Teo}

Now, following Hertling \cite{hertl3} (cf. M. Saito\cite{MSaito}), it is posible to choose a convenient basis of $\Omega_f$ such that their representatives whose principal terms into its asymptotical expansion are inside ${H}^{\prime\prime}_0\subset V^{>\alpha_1}\subset V^{>-1}$ form a Jordan basis for the nilpotent part of the monodromy. 

In fact, using the duality properties that the Deligne splitting inherits from the PMHS on $H^{n}(X_\infty,\QQ)$ we begin by choosing a $\mathbb{C}$-basis $A_1,\ldots,A_\mu$ of
$H^n(X_\infty,\mathbb{C})$ such that:
\begin{itemize}
\item[(a)] it corresponds with $s_1,\ldots,s_\mu$, $s_i\in C_{\alpha_i}$, $1\leq i\leq \mu $, by the relation

\begin{equation}\label{eqn:JordanBasisParaN}
\partial^{k_i}_ts_i=\psi (A_i),\hskip .5cm \hbox{ i.e. } L_{k_i}s_i=\psi^{-1}\partial^{k_i}_ts_i= A_i
\end{equation}

 where $k_i$ is such that $\alpha_i-k_i\in (-1,0]$, i.e., $s_i\in \HH_{k_i}$,
\item[(b)] $s_1,\ldots, s_\mu$
project onto a basis of $\bigoplus_{-1<\alpha<n}\,Gr^\alpha_V(\Hpp/\partial^{-1}_t\Hpp)\simeq \bigoplus_{-1<\alpha<n}\,Gr^\alpha_V(\Omega_f)$ satisfying that
\begin{equation}\label{eqn:dtNsj}
s_{\nu(i)}  = (Id-(\alpha_i+1)\partial_t^{-1})s_i=\partial^{-1}_tN\,s_i=N\partial^{-1}_t\,s_i
\end{equation}

where $\nu$ is such that

\begin{equation}\label{eqn:JordanBasis-N}
A_{\nu(j)}=\frac{-1}{2\pi \sqrt{-1}}N A_{j},\hskip 1cm A_{\mu+1}=0.
\end{equation}

Actually, we may suppose that $\nu$ coincides with $\nu_N$.
\item[(c)] there exists an involution $\kappa:\{1,\ldots,\mu\}\longrightarrow \{1,\ldots,\mu\}$
with $\kappa(i) = \mu+1-i$ if $\alpha_i\neq \frac{1}{2}(n-1)$ and $\kappa(i) = \mu+1-i$ or $\mu(i)=i$ otherwise and
satisfying the orthogonality relations
$\langle s_i,s_j\rangle_{\HH}=\delta_{\kappa(i),j}$.

Essentially this comes from the fact that
 \begin{equation}
\label{eqn:matrixQ}
\delta_{\kappa(i),l}= \left\{ \begin{array}{ll}
(-1)^{r_{l}}\Bigg(\displaystyle \frac{1}{2\pi\sqrt{-1}}\Bigg)^{n}\,Q\big(A_{i},A_{l}\big) & \textrm{if $\lambda_{\alpha_{i}}=\bar{\lambda}_{\alpha_{l}}\neq 1$}\\
\\
(-1)^{r_{l}+1}\Bigg(\displaystyle\frac{1}{2\pi\sqrt{-1}}\Bigg)^{n+1}\,Q\big(A_{i},A_{l}\big) & \textrm{if $\lambda_{\alpha_{i}}=\lambda_{\alpha_{l} }=1$,}
\end{array} \right.
\end{equation}
where $\alpha_{i},\alpha_{l}\in sp(f)$ and $\delta_{\kappa(i),l}$ is the Kronecker delta.
Hence the paring $\langle \, ,\, \rangle_{\HH}$ acquires the normal form $[\langle \, ,\, \rangle_{\HH}]=S_1$, where $S_1$ is  the anti-diagonal $\mu\times \mu$ matrix 
\[S_1=\left(
\begin{array}{ccccccc}
&    &   &  &        &&1  \\
&    &  &  &  &\cdots & \\
&    &&      &1        && \\
&&  & \left(
\begin{array}{cccc}
1 & 0 & \cdots &0\\
0 & 1 & \ddots&\vdots \\
\vdots & \ddots & \ddots & 0 \\
0 & \cdots & 0&1 \\
\end{array}
\right)
&  &  &\\
&       & 1  &  & & \\
& \cdots&  &  &  &  &\\
1&       &  &  &  &  &\\
\end{array}
\right).
\]
\item[d)] Following \cite[Lemma 5.2]{hertl3}, from $s_1,\ldots, s_\mu$ a $\mathbb{C}\{\{\dt^{-1}\}\}$-basis $h_1,\ldots,h_\mu$ for Brieskorn lattice $H^{\prime\prime}_0$ can be constructed
\begin{equation}\label{eqn:basesaito}
h_{i}\in\Hpp\cap (C_{\alpha_i}\oplus\sum_{\substack{j,p:\\\substack{ \alpha_{i}+p<\alpha_{j}\in sp(f)\\ p\geq 1}}}\mathbb{C}\cdot\dt^{p}s_{j}),\ \ i=1,\ldots,\mu;
\end{equation}
which induces the corresponding basis for $\HH^{\prime\prime}_0$ whose principal part of each $h_i$ is $s_i$. Finally, one  chooses forms $\eta_1,\ldots,\eta_\mu\in \Omega^{n+1}_{\mathbb{C}^{n+1},0}$ for which $h_1,\ldots,h_\mu$ are the corresponding images under $s$.  

\end{itemize}

 \subsection{The normal form of map multiplication  by $f$ in the Saito--Hertling basis}
 Following \cite[Proposition 5.4]{hertl3} we have a description for the map multiplication by $t$:
 $$t:\HH^{\prime\prime}_0\rightarrow \HH^{\prime\prime}_0$$
  that corresponds to the map multiplication by $f$ via the relation
 $s[f\bullet]=ts[\bullet]$.
  
 From this result we may obtain an induced normal form for the map multiplication by $f$, $$M_f:\Omega_f\longrightarrow\Omega_f, [\omega]\mapsto M_f[\omega]:=[f\omega],$$
 In fact, (choosing coordinates $(z_0,\ldots,z_n)$ on $(\mathbb{C}^{n+1},0)$) it will be determined by
$$[g dz_0\wedge dz_1\wedge \cdots \wedge dz_n]\mapsto [(fg) dz_0\wedge dz_1\wedge \cdots \wedge dz_n].$$
Hence, we may use that $A_f$ is isomorphic to $\Omega_f$, via the map $1\mapsto [dz_0\wedge dz_1\wedge \cdots \wedge dz_n],$ to obtain the following commutative diagram of $\mathcal{O}_{\mathbb{C}^{n+1},0}$-modules:
 $$\xymatrix @C=4pc @R=4pc{A_f\ar[r]^-{M_f}\ar[d]_-{\simeq}&A_{f}\ar[d]^-{\simeq}\\
\Omega_f\ar[r]^-{M_f}&\Omega_f
}$$
 and therefore we have up to this conjugation (and up to choosing holomorphic coordinates) the map given by \eqref{eqn:multiplicationbyf}. And therefore, we have that \begin{equation}\label{eqn:t-SaitoBasis}
f[\eta_i]=s^{-1}(th_{i})=s^{-1}(h_{\nu_{N}(i)})+\sum^{\mu_i}_{\substack{j=1}}c_{ij}s^{-1}(h_{j})\in s^{-1}(\mathbb{H}^{\prime\prime}_0/\HH^{\prime}_0)=\Omega_f,
\end{equation}
where $c_{ij}$ are constants determined by the spectral values in such a way that
\begin{displaymath}
c_{ij}:= \left\{ \begin{array}{ll}
0 &\textrm{if\quad$\alpha_{i}+1\geq\alpha_{i}$}\\
& \\
(\alpha_{j}-1-\alpha_{i})\;c^{(1)}_{ij} & \textrm{if\quad$\alpha_{i}+1<\alpha_{j}$}.
\end{array} \right.
\end{displaymath}
From \eqref{eqn:JordanBasis-N} we have that the $\mu_i\,'s$ depend on the length of the Jordan chain of $A_i$ with respect to $\nu_{N},$ that is, $\mu_i\leq\mu-\nu_{N}(i)<\mu$ for each $i=1,\ldots, \mu.$ 
Hence, this determines a $\mu\times\mu$ matrix $N_1$ using the terms of order greater than $\alpha_{i}+1$, $1\leq i\leq\mu$ according the expansions in \eqref{eqn:t-SaitoBasis} ; this is done by using the commutative diagram
\[
\xymatrix @R 5pc @C 5pc{\mathbb{C}^{\mu}\ar[r]&\mathbb{C}^{\mu}\\
	\frac{\mathbb{H}^{\prime \prime}_0}{\mathbb{H}^{\prime}_0}\ar[u]_-{[\underline{h}]}^-{\simeq}\ar[r]& \frac{\mathbb{H}^{\prime \prime}_0}{\mathbb{H}^{\prime}_0}\ar[u]^-{[\underline{h}]}_-{\simeq}
}
\]
given by $[\underline{h}]$-coordinates, in such a way that $N_1$ is equals to the transpose matrix of $(c_{ij}).$

Let us use the notation $K_f:=\langle \, ,\, \rangle_{\HH}$.

\begin{Def}
	\label{def:Atop}
	Let ${N}_{top}$ be the $\mu\times\mu$ matrix given by \[{N}_{top}=\big[t\big]_{[\underline{h}]}-N_1, \]
	where $[t]_{[\underline{h}]}$ is the induced $\mathbb{C}$-basis for $\dfrac{\mathbb{H}^{\prime \prime}_0}{\mathbb{H}^{\prime}_0}$
	coming from basis $\{h_j\}.$ Define the endomorphism
	$\mathbf{N}_1:H^{n}(X_\infty,\mathbb{C})$ such tat its $\underline{A}$-matrix expression is given by $N_1.$
  
\end{Def}

\begin{Lemma}\label{lem:marc1}
	Let $1\leq i,j\leq\mu$. For any integers $p,q$ such that $0\leq p\leq r_{i}$ and $q\leq r_{j}$,
	\begin{eqnarray*}
		K_f(\dt^{p}s_{i},\dt^{q}s_{j})&=&(-1)^{-q}K_f(s_{i},s_{j})\cdot\dt^{p+q}\\
		&=& (-1)^{-q}\delta_{\kappa(i),j}\cdot\dt^{-n-1+p+q}.
	\end{eqnarray*}
	where $r_{i},r_{j}$ are the levels of $s_{i},s_{j}$, respectively.
\end{Lemma}
\begin{proof}
	From definition of $K_f$,
	\begin{equation}\label{eq:1}
	K_f(s_{i},s_{j})=(-1)^{r_{j}} K_f(\dt^{r_{i}}s_{i},\dt^{r_{j}}s_{j})\cdot\dt^{-r_{i}-r_{j}},
	\end{equation}
	where $r_{i},r_{j}$ are the levels of $s_{i},s_{j}$, respectively. In the same way,
	\begin{eqnarray}\label{eq:2}
	K_f(\dt^{p}s_{i},\dt^{q}s_{j})&=&(-1)^{r_{j}-q} K_f(\dt^{r_{i}-p}\dt^{p}s_{i},\dt^{r_{j}-q}\dt^{q}s_{j})\cdot\dt^{-r_{i}-r_{j}+p+q}\nonumber\\
	&=&(-1)^{r_{j}-q} K_f(\dt^{r_{i}}s_{i},\dt^{r_{j}}s_{j})\cdot\dt^{-r_{i}-r_{j}+p+q}
	\end{eqnarray}
	Finally, the claim follows from~(\ref{eq:1}),~(\ref{eq:2}), and~(\ref{eqn:matrixQ}).
\end{proof}

Let $\big[res_{f,0}(\bullet,\bullet)\big]_{[\underline{\eta}]}$ be the $[\underline{\eta}]$-matrix of the Grothendieck pairing which is induced by \eqref{eqn:basesaito}.

\begin{Prop}\label{prop:SaitoRes-hbase-sbase}
	\begin{enumerate}
		\item For any $i,j\in\{1,\ldots,\mu\}$,
		\begin{eqnarray}\label{eqn:PShihj2}
		K_f(h_{i},h_{j})&=&\delta_{\kappa(i),j}\cdot\dt^{-n-1}\in\mathbb{C}\cdot\dt^{-n-1}\nonumber\\
		&=&K_f(s_{i},s_{j}).
		\end{eqnarray}
		In particular,
		\begin{equation}
		\label{eqn:PShihj2_2}
		\big[res_{f,0}(\bullet,\bullet)\big]_{\underline{[\eta]}}=S_1.
		\end{equation}
		\item Suppose that $h_{\nu(i)}\neq 0$. Then, for any $j\in\{1,\ldots,\mu\}$,
		\begin{eqnarray}\label{eqn:PShihj3}
		K_f(h_{\nu(i)},h_{j})&=&K_f(s_{\nu(i)},s_{j})\nonumber\\
		&=&K_f(\dt^{-1}\widetilde{N}_{\alpha_{i}}s_{i},s_{j}).
		\end{eqnarray}
		\item The $(i,j)$ entry of matrix $[N]^{T}_{\underline{A}}S_1$ is equals to
		\begin{equation}
		\dt^{n+1}K_f(h_{\nu(i)},h_{j}).
		\end{equation}
		Or, equivalently,
		\begin{equation}
		\label{eqn:PShihj4}
		{N}_{top}=[N]_{\underline{A}},
		\end{equation}
		where $\big[N\big]_{\underline{A}}$ be the $\mu\times\mu$ constant $\underline{A}$-matrix associated to the operator $N$ satisfying \eqref{eqn:JordanBasis-N} and \eqref{eqn:matrixQ}.
		
	\end{enumerate}
\end{Prop}
\begin{proof}
	Note that (\ref{eqn:PShihj3}) follows from (\ref{eqn:PShihj2}) and \eqref{eqn:dtNsj}. Also, \eqref{eqn:PShihj2_2} follows from \eqref{eqn:PShihj2} and (ii) in Theorem~\ref{thm:varchenko}. On the other hand, (3) follows from \eqref{eqn:PShihj3} and by noting that
	\begin{equation*}
	K_f(s_{\nu(i)},s_{l})=\left\{ \begin{array}{ll}
	(-1)^{r_{l}}\displaystyle\Bigg(\frac{1}{2\pi\sqrt{-1}}\Bigg)^{n}Q(NA_{i},A_{l})\cdot\dt^{-n-1} & \textrm{if ($\alpha_{i}-r_{i})+(\alpha_{l}-r_{l})=-1$}\\
	\\
	\\
	(-1)^{r_{l}+1}\cdot\Bigg(\displaystyle\frac{1}{2\pi\sqrt{-1}}\Bigg)^{n+1}Q(NA_{i},A_{l})\cdot\dt^{-n-1} & \textrm{if $(\alpha_{i}-r_{i})=(\alpha_{l}-r_{l})=0$}.\\
	\end{array} \right.
	\end{equation*}
	Finally, we will prove~\eqref{eqn:PShihj2}. Let $1\leq i,l\leq\mu$. Assuming
	\[h_{i}=s_{i}+\sum^{\mu}_{\substack{j=1\\p\geq 1\\ \alpha_{i}+p<\alpha_{j}}} c^{(p)}_{ij} \cdot\dt^{p}s_{j} \quad\textrm{ and }\quad h_{l}:=s_{l}+\sum^{\mu}_{\substack{l_{1}=1\\q\geq 1\\ \alpha_{l}+q<\alpha_{l_{1}}}} c^{(q)}_{ll_{1}} \cdot\dt^{q}s_{l_{1}},\]
	it follows that
	\begin{eqnarray*}
		K_f(h_{i},h_{j})&=&K_f\bigg(s_{i}+\sum^{\mu}_{\substack{j=1\\p\geq1\\ \alpha_{i}+p<\alpha_{j}}} c^{(p)}_{ij} \cdot\dt^{p}s_{j}\,,\,s_{l}+\sum^{\mu}_{\substack{l_{1}=1\\q\geq1\\ \alpha_{l}+q<\alpha_{l_{1}}}} c^{(q)}_{ll_{1}} \cdot\dt^{q}s_{l_{1}}\bigg)\\
		&=&K_f(s_{i},s_{l})+\sum^{\mu}_{\substack{j=1\\p\geq1\\ \alpha_{i}+p<\alpha_{j}}}\sum^{\mu}_{\substack{l_{1}=1\\q\geq1\\ \alpha_{l}+q<\alpha_{l_{1}}}} c^{(p)}_{ij}\cdot c^{(q)}_{ll_{1}} \cdot K_f\bigg(\dt^{p}s_{j}\,,\,\dt^{q}s_{l_{1}}\bigg).\\
	\end{eqnarray*}
	So, by equation~\eqref{eqn:matrixQ} and  Lemma~\ref{lem:marc1}
	\begin{eqnarray}\label{eqn:proofPropPs}
	K_f(h_{i},h_{j}) &=&\delta_{\kappa(i),l}\cdot\dt^{-n-1}+{}\nonumber\\
	&& {}+\sum^{\mu}_{\substack{j=1\\p\geq1\\ \alpha_{i}+p<\alpha_{j}}}\sum^{\mu}_{\substack{l_{1}=1\\q\geq1\\ \alpha_{l}+q<\alpha_{l_{1}}}} c^{(p)}_{ij}\cdot c^{(q)}_{ll_{1}} \cdot (-1)^{q} \delta_{\kappa(j),l_{1}}\cdot\dt^{-n-1+p+q}.
	\end{eqnarray}
	Since $1\leq p,q\leq n$, it satisfies that $-(n-1)\leq n+1-p-q\leq n-1$. So, by (i) in Theorem~\ref{thm:varchenko}, it follows that the second summand of \eqref{eqn:proofPropPs} is equal to zero.  This complete the proof.
\end{proof}

Proposition \ref{prop:SaitoRes-hbase-sbase} implies the following corollary.

\begin{Cor}\label{eq:matrixMf}
The matrix expression of $f$ in the basis $[\eta_1],\ldots,[\eta_\mu]\in \Omega_f$ is
$$[M_f]_{\underline{\eta}}=[N]_{\underline{A}}+N_1=N_{top}+N_1,$$ where $N$ is in canonical Jordan Form with respect to basis $\underline{A}=\{A_j\}$ and $N_1=(c_{ij})^{tr}$ accordingly to \eqref{eqn:JordanBasisParaN} and \eqref{eqn:t-SaitoBasis}  .
\end{Cor}

Explicitly, we have the following normal form for the map multiplication by $f$:
\[[M_f]_{\underline{\eta}}=\begin{pmatrix}
	0&0&\cdots&0&0\\
	\hdotsfor[2]{5}\\
	n_{\nu(1)1}&0&&0&0\\
	\hdotsfor[2]{5}\\
	0&n_{\nu(2)2}&&0&0\\
	\hdotsfor[2]{5}\\
	0&0&\ldots&n_{\nu(\mu-1) \mu-1}&0\\
	\hdotsfor[2]{5}\\
	0&0&\ldots&0&n_{\nu(\mu) \mu}\\
	\hdotsfor[2]{5}\\
	c_{11}&c_{21}&\ldots&c_{\mu-1 1}&c_{\mu 1}\\
	c_{12}&c_{22}&\ldots&c_{\mu-1 2}&c_{\mu 2}\\
	\hdotsfor[2]{4}\\
	c_{1\mu_1}&c_{2\mu_2 }&\ldots&c_{ \mu \mu_{\mu-1}}&c_{\mu \mu_{\mu}}\\
\end{pmatrix}:=[N]_{\{A_j\}}+\mathbf{N}_1,\]
con
\[n_{\nu(i)i}=\begin{cases}
	1 \hskip .4 cm\textrm{si}\hskip .4 cm \nu(i)\neq \mu+1\\
	0  \hskip .4 cm\textrm{si}\hskip .4 cm \nu(i)= \mu+1.\\
\end{cases}
\]

\subsection{Main results}
 
   \begin{Teo}[First main result]\label{thm:firstpartthm}
   	There exist an isomorphism of $\mathbb{C}$--vector spaces $\varphi:\Omega_f\rightarrow H^{n}(X_{\infty},\mathbb{C})$ and an automorphism $J:H^{n}(X_{\infty},\mathbb{C})\rightarrow H^{n}(X_{\infty},\mathbb{C})$
   	   	such that:
  	\begin{enumerate}
  		\item[(a)] The bilinear forms $res_f$ on $\Omega_f$ and $Q(\bullet,J\bullet)$ on $H^n(X_{\infty},\mathbb{C})$ are equivalent, that is,
  		
  		$$res_f(\bullet,\bullet)=Q(\varphi\bullet,J\varphi \bullet).$$
  	\item[(b)] For $j=1,\ldots,n$, the $\mathbb{C}$-bilinear spaces 
  	\begin{align*}
  	(\Omega_f, res(f^{j}\bullet,\bullet))\mbox{ and } \biggl(H^{n}(X_\infty,\mathbb{C}), Q\bigl((N+\mathbf{N}_1)^{j}\bullet,J\bullet\bigr)\biggr)	
  	\end{align*}
  	are equivalent, that is,
  	$$res(f^{j}\bullet,\bullet)=Q\bigl((N+\mathbf{N}_1)^{j}\varphi\bullet,J\varphi\bullet\bigr).$$
  	\item[(c)] If $f$ has finite monodromy, then for $j=1,\ldots,n$, the $\mathbb{C}$-bilinear spaces 
  	\begin{align*}
  		(\Omega_f, res(f^{j}\bullet,\bullet)) \mbox{ and } \biggl(H^{n}(X_\infty,\mathbb{C}), Q\bigl((\mathbf{N}_1)^{j}\bullet,J\bullet\bigr)\biggr)
  	\end{align*}
  	are equivalent, that is,
  	$$res(f^{j}\bullet,\bullet)=Q\bigl(\mathbf{N}_1^{j}\varphi\bullet,J\varphi\bullet\bigr).$$
\end{enumerate}
    \end{Teo}

\begin{proof}
	From Proposition \ref{prop:SaitoRes-hbase-sbase} one has the orthogonality relations
	\begin{align}
		\label{eqn:lemaOrthSaitoBasis}
		\langle s_j,s_{\ell}\rangle_{\HH}=&\delta_{\kappa(j),\ell}=\langle [h_j],[h_{\ell}]\rangle_{\HH}=res_f([\eta_j],[\eta_{\ell}])\nonumber\\
		\langle s_{\nu_{N}(j)},s_{\ell}\rangle_{\HH}=&\delta_{\kappa(\nu_{N}(j)),\ell}=\langle [h_{\nu_{N}(j)}],[h_{\ell}]\rangle_{\HH}= res_f([\eta_{\nu_{N}(j)}],[\eta_{\ell}])
	\end{align}
	
	By definition of $\langle,\rangle_{\HH}$ on pairs $$(s_j,s_\ell)\in Gr_V^{\alpha_j}\times Gr_V^{-1-\alpha_j+n}=C_{(\alpha_j-k_j)+k_j}\times C_{-1-(\alpha_j-k_j)+(n-k_j)}=C_{\beta_j+k_j}\times C_{-1-\beta_j+(n-k_j)},$$ where $\beta_j=\alpha_j-k_j\in (-1,0]\cap \mathbb{Q}$, we have
	\begin{align*}
		\delta_{\kappa{(j)},\ell}=\langle s_j,s_\ell\rangle_{\HH}=&\displaystyle\frac{(-1)^{1+\lfloor\beta_j\rfloor}}{(2\pi i)^{n+1+\lfloor\beta_j\rfloor}}Q(L_{k_j}s_j,(-1)^{n-k_j}L_{n-k_j}s_\ell))\\
		=&\frac{(-1)^{1+\lfloor\beta_j\rfloor}}{(2\pi i)^{n+1+\lfloor\beta_j\rfloor}}Q(A_j,(-1)^{n-k_j}A_\ell).
	\end{align*}
	Therefore, we can be defined on each generalized eigenspace $H^{n}(X_\infty,\mathbb{C})_{e^{-2\pi i \alpha_j}}$ an automorphism $J_j$ by the condition
	$$J_j(A_{j}):=(-1)^{n-k_j}\frac{(-1)^{1+\lfloor\beta_j\rfloor}}{(2\pi i)^{n+1+\lfloor\beta_j\rfloor}}A_{j}\hskip 1cm \lfloor\beta_j\rfloor=-1,0.$$
	Define an automorphism in $H^n$ as the direct sum $J:=\oplus^{\mu}_{i=1}J_j$.
	The automorphism  $J$ does not depend on the basis $A_1,\ldots,A_\mu$ but only on the Hodge filtration $F$. Indeed, for each $\alpha_j\notin \ZZ$, by definition $A_j\in F^{p(j)}(H^n_{e^{-2 \pi i \alpha_{j}}})$ if and only if $s_j\in\frac{(V^{\alpha_j}\cap \partial_t^{n-p(j)}H_0^{\prime\prime})+V^{>\alpha_j}}{V^{>\alpha_j}}=\partial^{n-p(j)}_tGr^{\alpha_j}_V\Hpp$, and
	the fact that the $s_1,\ldots,s_\mu$ projects onto a basis of $Gr^{*}_V(H^{\prime\prime}_0/\partial^{-1}_tH^{\prime\prime}_0)$ implies that $Gr^{p(j)}_F(J)=(-1)^{n-k_j}\frac{1}{(2\pi i)^{n}}I$ and
	$k_j=n-p(j)$. Then, automorphism $J$ is given on $F^{p(j)}(H^{n}_{e^{-2\pi i \alpha_j}})\setminus F^{p(j)+1}(H^{n}_{e^{-2\pi i \alpha_j}})$ simply as $ J=(-1)^{p(j)}\frac{1}{(2\pi i)^{n+1}}I$. Similarly, for $\alpha_j\in \ZZ$, $J$ is given on $F^{p(j)}(H^{n}_{1})\setminus F^{p(j)+1}(H^{n}_{1})$ by $J_j=(-1)^{p(j)+1}\frac{1}{(2\pi i)^{n+1}}I$. Automorphism $J$ is such that $Q(\bullet,J\bullet)$ is symmetric on $H^{n}(X_\infty,\mathbb{C})$.
	
	Theorem \ref{thm:varchenko}, \eqref{eqn:t-SaitoBasis} and \eqref{eqn:lemaOrthSaitoBasis} prove the theorem by
	choosing bases $\{[\eta_j]:j=1,\ldots,\mu\}$ for $\Omega_f$ and $\{A_j:j=1,\ldots,\mu\}$ for $H^n(X_{\infty},\mathbb{C})$, and 	the
	isomorphism $\xymatrix{\Omega_f\ar[r]^-{\varphi}&H^{n}(X_{\infty},\mathbb{C})}$ is  given by
	\begin{equation}\label{eqn:DefnIsoUpsilon}
		\xymatrix@R=1pc{\Omega_f\ar[r]^-{s}&\dfrac{\mathbb{H}^{\prime \prime}_0}{\mathbb{H}^{\prime}_0}\ar[r]&H^{n}(X_{\infty},\mathbb{C})\\
			[\eta_{j}]\;\ar@{|->}[r]&[h_{j}]\;\ar@{|->}[r]&A_j,}
	\end{equation}
	for each $1\leq j\leq\mu$.	
\end{proof}

This theorem endows the higher bilinear forms $res_{f,0}(f^j\bullet,\bullet)$ with a normal form given by the right expressions. This will be illustrated in section \ref{sec:resMf} and the examples below.

\begin{Cor}[Second main result]\label{cor:mainthm}
	The Jordan Chains of multiplication by $f$ in $\Omega_f$ are obtained by binding of Jordan chains of $N$.
	\end{Cor}
\begin{proof}
	In order to simplify the notation, from now on we set $\nu=\nu_N$ as in \eqref{eqn:JordanBasis-N}.
	
	Take a Jordan chain adapted to the wight filtration $W_{\bullet}(f)\subset\Omega_f\simeq A_f$ associated to the nilpotent operator multiplication by  $f$ (see subsection \ref{sec:AfMf}):
	
	$$
	\xymatrix{
		\frac{\Omega^{n+1}_{\mathbb{C}^{n+1},0}}{df\wedge \Omega^{n}_{\mathbb{C}^{n+1},0}}\supset C(v_{j}):
		\ar[r] [v]\ar[r]& [fv]\ar[r]& [f^2 v]\ar[r] & [f^3v]\ar[r]&\cdots\ar[r]& [f^{\ell-1}v]\ar[r]& [f^\ell v]\ar[r]&0.
	}
	$$
	We can see this Jordan chain via the isomorphism $s$ as
	
	$$
	\xymatrix{
		C(v):
		\ar[r] v+df\wedge \Omega^{n}_{\mathbb{C}^{n+1},0} \ar[r]& fv+df\wedge \Omega^{n}_{\mathbb{C}^{n+1},0}\ar[r]& f^2 v+df\wedge \Omega^{n}_{\mathbb{C}^{n+1},0}\ar[r]\cdots& f^\ell v + df\wedge \Omega^{n}_{\mathbb{C}^{n+1},0}\ar[r]&0\\
		s(C(v)):\ar[r] s[v]_{0}+{\mathbb{H}^{\prime}_0}\ar[r]& s[fv]_{0}+{\mathbb{H}^{\prime}_0}\ar[r]& s[f^2 v]_{0}+{\mathbb{H}^{\prime}_0} \cdots\ar[r]& s[f^\ell v]_{0}+{\mathbb{H}^{\prime}_0}\ar[r]&0.
	}
	$$
		
	If we assume that $\underline{v}$ is a Jordan basis of $(\Omega_f,f)$ we may follow
	its evolution with respect to the Saito-Hertling basis $$\underline{h}:=h_1,\ldots, h_\mu$$
	Since, 
	\begin{align}
		\label{eqn:SHbasis-t}
		h_\ell&=s_\ell+\sum_{\substack{j,p\\p\geq 1\\\alpha_l+p<\alpha_j}}c^{(p)}_{\ell j}\dt^{p}s_{j},\nonumber\\
		&\\
		th_\ell&=h_{\nu(\ell)}+\sum_{\substack{j\\\alpha_l+1<\alpha_j}}c^{(1)}_{\ell j}(\alpha_j-\alpha_\ell-1)h_j,\ \  \ell=1,\ldots,\mu.\nonumber
	\end{align}
	In fact, to simplify the idea we will suppose there is a Jordan chain for $M_f$ with length $3$:
	\begin{align*}
		C(v):\;s[v]_{0}\longrightarrow s[fv]_{0}\longleftarrow s[f^{2}v]_{0}\longrightarrow 0.
	\end{align*}
	Since $s[v]_0\in \frac{\Hpp}{\dt^{-1}\Hpp}$, there are unique constants $a_{1},\ldots,a_{\mu}\in \mathbb{C}$ such that $s[v]_0$ is an $\mathbb{C}$-linear combination
	$$s[v]_0=a_{1}h_1+\cdots+a_{\mu}h_\mu.$$
	Without loss of generality, we may keep with the summands that effectively contribute to the $\underline{h}$-linear expansion, that is,  we do not contemplate the constants $a_j=0$, and therefore, we assume that $a_{1},\ldots,a_{r}\neq0$, $1\leq r\leq \mu$ and
	
	\begin{equation}\label{eqn:paso0}
		s[v]_0=a_{1}h_{1}+\cdots+a_{r}h_{r}.
	\end{equation} 
	
	This linear combination will denote the first \textbf{Step 0} which encodes the beginnings of topological chains according to \eqref{eqn:SHbasis-t}, that is, we consider
	\begin{equation}
		\label{eqn:beginchains}
		s_{{1}},\ldots, s_{{r}}\,,
	\end{equation}
	which result from not to apply $\{\mathbf{f}\}$, that is, we apply the identity map  $\{\mathbf{f}\}^{0}:=I$ to $s[v]_0$. 
	
	Analyzing carefully {\bf Step 0}, one has that the begins of Jordan $f$-chains in  \eqref{eqn:beginchains} are already descended from someone, and therefore they may be of the following types:
	\begin{itemize}
		\item[i)] $s_{{j}}$ with  $A_j\in Im\, N\setminus Ker\,N$ is a bind of a chain:
		$$
		C(N):\bullet\rightarrow\bullet\cdots \rightarrow A_j\rightarrow\cdots\rightarrow\bullet\rightarrow 0	 
		$$
		\item[ii)] $s_{{j}}$ with  $A_j\in Im\, N\cap Ker\,N$:
		$$
		C(N):\bullet\rightarrow\bullet\cdots \rightarrow A_j\rightarrow 0.	 
		$$
		\item[iii)] $s_{{j}}$ with $A_j\in Ker\, N\setminus Im\, N$: 
		$$
		C(N): A_j\rightarrow 0.	 
		$$
		\item[iv)] $s_{{j}}$ with $A_j\notin Ker\, N \cup Im\,N$ that means that the $A_j$ is a begin of a $N$-chain:
		$$
		C(N):A_j\rightarrow\bullet\cdots \rightarrow \cdots\rightarrow\bullet\cdots \rightarrow 0.	 
		$$
	\end{itemize}
	
	On the other hand, {\bf Step 1} will codify the fact to kill the elements of $\ker(f)\simeq \ker(t)$  which appear in {\bf Step 0}. In fact, from \eqref{eqn:SHbasis-t} we assume that we have only an expansion with basis elements $h_{i_\ell}\notin\ker(t)$, $\ell=1,\ldots,r_0\leq r$ with the non zero constants $a_{i_1},\ldots,a_{i_{r_0}}$. Without loss of generality assume that $r_0=r$ and ${i_1}<\ldots<{i_{r_0}}$ whose order respects the ordered basis $\underline{h}$ as in \eqref{eqn:SHbasis-t}, are such that 
	\begin{align*}
		s[fv]_0=&a_{i_{1}}th_{i_{1}}+\cdots+a_{i_{ r}}th_{\ell_r}\\
		=&a_{i_1}h_{\nu(i_1)}+a_{i_1}h_{\nu(i_2)}+\cdots+a_{i_r}h_{\nu(i_r)}+\sum_{i_1,\ldots,i_\mu}\sum_{j(i_k)}\, d_{(i_1,\ldots,i_r)}\cdot h_{j(i_k)}\\
		=&a_{i_1}\left(s_{\nu(i_1)}+\sum_{\substack{j(i_1),p_1\\p_1\geq 1\\\alpha_{\nu(i_1)}+p_1<\alpha_{j(i_1)}}}c^{(p_1)}_{\nu(i_1) j(i_1)}\dt^{p_1}s_{j(i_1)}\right)
		+\cdots+a_{i_r}\left(s_{\nu(i_r)}+\sum_{\substack{j(i_r),p_r\\p(i_r)\geq 1\\\alpha_{\nu(i_r)}+p_r<\alpha_{j(i_r)}}}c^{(p_r)}_{\nu(i_r) j(i_r)}\dt^{p_r}s_{j(i_r)}\right)\\
		&+\sum_{i_1,\ldots,i_\mu}\sum_{j(i_k)}\, d_{(i_1,\ldots,i_r)}\cdot h_{j(i_k)} \\
		=&\sum_{\ell=1}^{r} a_{i_\ell} s_{\nu(i_\ell)} +\sum_{\ell=1}^{r}\sum_{\mathbf{Cdn}(i_\ell)} a_{i_\ell}c^{(p_\ell)}_{\nu(i_\ell) j(i_\ell)}\dt^{p_\ell}s_{j(i_\ell)}+\sum_{i_1,\ldots,i_\mu}\sum_{j(i_k)}\, d_{(i_1,\ldots,i_r)}\cdot h_{j(i_k)},
	\end{align*}
	where
	$\mathbf{Cdn}(i_\ell)$ means the additive condition $${\substack{j(i_\ell),p_\ell\\p_\ell\geq 1\\\alpha_{\nu(i_\ell)}+p_\ell<\alpha_{j(i_\ell)}}},\ \ \ell=1,\ldots,r,$$
	and the $d_{(i_1,\ldots,i_r)}$ are constants and each subindex $j(i_1),\ldots,j(i_r)$ denotes the dependence on $i_1,\ldots,i_r$, according to \eqref{eqn:SHbasis-t}, respectively.
		Continuing in this manner we prove the corollary.
\end{proof}

\subsection{ Canonical form for The Bilinear Form $res_f(M_f\bullet,\bullet)$}\label{sec:resMf}

We will explain a more details for the bilinear form of order $j=1$ according to Theorem \ref{thm:firstpartthm}. On  $H^n(X_\infty,\mathbb{C})$ one has a canonical (polarized) mixed Hodge structure whose polarization is induced by $Q_{X_t}$, that is, $Q\simeq Q_{X_t}$ and therefore, we have the induced orthogonal decomposition
\begin{equation}\label{eqnocXt1}
H^n(X_\infty,\mathbb{C}) = H^n(X_\infty,\mathbb{C})_1 \bigoplus H^n(X_\infty,\mathbb{C})_{ 1}^{\perp_Q}, 
\end{equation} 
$$ H^n(X_\infty,\mathbb{C})_{ 1}^{\perp_Q}:=H^n(X_\infty,\mathbb{C})_{-1} \bigoplus [\bigoplus_{Im \lambda>0} H^n(X_\infty,\mathbb{C})_{\lambda,{\overline{\lambda}}}]$$
which is defined over $\QQ$.

From results in Theorem \ref{thm:firstpartthm} we may interpret the mixed polarized Hodge structure as the one simply described
using the basis of Saito-Hertling. Associated to the function $f$ there is its spectrum,
$-1<\alpha_1\leq\ldots\leq\alpha_{\mu} < n $ who are logarithms of the eigenvalues of
the monodromy and a basis $A_1,\ldots, A_\mu$ of $ H^n(X_\infty,\QQ)$
such that $A_j$ is an eigenvector of $M_s$ with eigenvalue $e^{-2\pi i \alpha_j}$
forming a Jordan basis for the nilpotent operator $N$ (i.e. $N(A_j)=A_{\nu(j)}$)
and puts the
bilinear form $Q$ into canonical form: In fact, 
organizing the basis $\{A_j\}$ so that first
we put those corresponding to $\alpha_j$ an integer, and then the rest,
the matrix expression of $Q$ in this basis is, for $n$ even or odd, respectively:
\begin{align}
Q=\begin{pmatrix}( 2\pi i )^{n+1}\begin{pmatrix}0 & 1 \cr -1 & 0
\end{pmatrix} & 0 \cr 0 & (2 \pi i )^{n}\begin{pmatrix}0 && 1 \cr &I&\cr 1 && 0
\end{pmatrix}
\end{pmatrix},\mbox{ or } \
Q=\begin{pmatrix} (2 \pi i )^{n+1} \begin{pmatrix}0 & &1 \cr & I & \cr 1 & &0
\end{pmatrix} & 0 \cr 0 &  (2 \pi i )^{n} \begin{pmatrix}0 & 1 \cr -1 & 0
\end{pmatrix}
\end{pmatrix},
\end{align}
being the bilinear form $Q$  $(-1)^{n+1}$-symmetric in the first factor term of \eqref{eqnocXt1} and
$(-1)^{n }$-symmetric in the other $Q$-orthogonal factor. 
And the endomorphism $N$ has the normal form:

\[[N]_{\underline{A}}=\begin{pmatrix}
	0&0&\cdots&0&0\\
	\hdotsfor[2]{5}\\
	n_{\nu(1)1}&0&&0&0\\
	\hdotsfor[2]{5}\\
	0&n_{\nu(2)2}&&0&0\\
	\hdotsfor[2]{5}\\
	0&0&\ldots&n_{\nu(\mu-1) \mu-1}&0\\
	\hdotsfor[2]{5}\\
	0&0&\ldots&0&n_{\nu(\mu) \mu}\\
	\hdotsfor[2]{5}\\
	0&0&\ldots&0&0\\
	0&0&\ldots&0&0\\
	\hdotsfor[2]{4}\\
	0&0&\ldots&0&0\\
\end{pmatrix}\]
where
\[n_{\nu(i)i}=\begin{cases}
	1 \hskip .4 cm\textrm{si}\hskip .4 cm \nu(i)\neq \mu+1\\
	0  \hskip .4 cm\textrm{si}\hskip .4 cm \nu(i)= \mu+1,\\
\end{cases}.
\]

where the columns with a one generate the Image of $[N]$ en the cero columns generate $\ker[N].$

The Hodge flags $F^kH^n(X_\infty,\mathbb{C})_1$ and $F^kH^n(X_\infty,\mathbb{C})_{\neq 1}$ consists of those vector spaces whose basis is the set of  $A_j$ with
$\alpha_j>k$, and the weight filtration comes from  the Jordan block structure of $N$ codified in $\nu$.

On the other hand, the Saito-Hertling basis comes also with the selection of $\eta_1,\ldots,\eta_\mu \in \Omega^{n+1}_{\mathbb{C}^{n+1},0}$ satisfying  that the principal
term in the asymptotic expansion of $s(\eta_j)$ is  $ \partial^{-k_i}_tt^{\alpha_j + N}A_j$ corresponding to the flat section $$\left(((\alpha_j-k_j+1)I-\frac{1}{2 \pi i}N)\cdots ((\alpha_j-k_j+k_j)I-\frac{1}{2 \pi i}N)\right)^{-1}A_j\in H^{n}(X_{\infty},\mathbb{C})_{e^{-2\pi i \alpha_j}}$$ 
where $k_j\in \mathbb{Z}_>0$ and $\alpha_j-k_j=\beta_j\in(0,-1]\cap\mathbb{Q}$
(see equations \eqref{eqn:dt} and \eqref{eqn:JordanBasisParaN}). Its classes $\{[\eta_j]\}$ form a basis of the Jacobian module $\Omega_f$.
In the basis $\{[\eta_j]\}$,
Grothendieck bilinear form receives the corresponding expressions:

$$[res_{f,0}]=\begin{pmatrix}\begin{pmatrix}0 & 1 \cr  1 & 0
\end{pmatrix} & 0 \cr 0 & \begin{pmatrix}0 && 1 \cr &I&\cr 1 && 0
\end{pmatrix}
\end{pmatrix}
\hskip 5mm,\hbox{or}\hskip 5mm
[res_{f,0}]=\begin{pmatrix} \begin{pmatrix}0 & &1 \cr & I & \cr 1 & &0
\end{pmatrix} & 0 \cr 0 &   \begin{pmatrix}0 & 1 \cr  1 & 0
\end{pmatrix}
\end{pmatrix}$$
If we denote by $S$ this matrix, then the relation between $S$ and $Q$ can be expressed as $S= QJ$ where the matrix of authomorphism  $J$ becomes of the form:
$$[J]=\begin{pmatrix}
J_1&0\cr 0 &J_{\neq1}
\end{pmatrix}:=\begin{pmatrix}\frac{1}{(2\pi i)^{n+1}}\begin{pmatrix}1 & 0 \cr  0 & -1
\end{pmatrix} & 0 \cr 0 &\frac{1}{(2\pi i)^{n}} \begin{pmatrix}1 & 0\cr 0 & 1
\end{pmatrix}
\end{pmatrix}
\hbox{ or }
[J]=\begin{pmatrix}\frac{1}{(2\pi i)^{n+1}}\begin{pmatrix}1 & 0\cr 0 & 1
\end{pmatrix} & 0 \cr 0 & \frac{1}{(2\pi i)^{n}}\begin{pmatrix}
1 & 0 \cr  0 & -1
\end{pmatrix}
\end{pmatrix}
$$
and  $J_1$ ( resp. $J_{\neq 1}$) is $\frac{1}{(2\pi i)^{n+1}}Id$ (resp. $\frac{1}{(2\pi i)^{n}}Id$ )  on the basis elements corresponding to even elements of the Hodge flag, and $J_1$ ( resp. $J_{\neq 1}$) is $\frac{-1}{(2\pi i)^{n+1}}Id$ (resp. $\frac{-1}{(2\pi i)^{n}}Id$ ) on
the basis elements corresponding to odd elements.

From Theorem \ref{thm:firstpartthm}, for any $u\in H^{n}(X_{\infty},\mathbb{C})$, $[u]\in\mathbb{C}^{\mu}$ denotes the $\underline{A}$-coordinate column vector associated to $u$. Hence, given $u\in H^{n}(X_{\infty},\mathbb{C})$, there exists a unique $\omega:=\varphi^{-1}(u)\in\Omega_f$ such that,  in term of coordinates, $[u]=[\omega]\in\mathbb{C}^{\mu}$.
Then, we have the following bilinear forms $$\xymatrix{B^{top},B^{alg}:H^{n}(X_{\infty},\mathbb{C})\times H^{n}(X_{\infty},\mathbb{C})\ar[r]&\mathbb{C}}$$
defined respectively by
\begin{equation}\label{eqn:B0}
	\xymatrix@R=1pc{(u,v)\;\ar@{|->}[r]&[u]^{T}\cdot(N_{top})^{T} S J\cdot [v],
	}
\end{equation}
\begin{equation}\label{eqn:B1}
	\xymatrix@R=1pc{(u,v)\;\ar@{|->}[r]&[u]^{T}\cdot(N_1)^{T} S J\cdot [v].
	}
\end{equation}

What Theorem \ref{thm:firstpartthm} says is that the bilinear form (of order one) is described in terms of the bilinear forms $B^{top},B^{alg}$ in such a way that for any $\omega,\eta\in\Omega_f$:

\begin{equation*}
res_{f,0}\big(f\omega\,,\,\eta\big)=B^{top}\big(\varphi(\omega),\varphi(\eta)\big)+B^{alg}\big(\varphi(\omega),\varphi(\eta)\big).
\end{equation*}

Equivalently, for any $u,v\in H^{n}(X_{\infty},\mathbb{C})$:
\begin{equation}\label{eqn:ordenuno}
	res_{f,0}\big(f\,\varphi^{-1}(u)\,,\,\varphi^{-1}(v)\big)=B^{top}(u,v)+B^{alg}(u,v).
\end{equation}

The bilinear form $B^{top}$ (resp. $B^{alg}$) will be called the topological (resp. algebraic) part of the bilinear form.

Let $\xymatrix{Gr_{V}\Omega_f\ar[r]&H^{n}(X_{\infty},\mathbb{C})}$ be the $\mathbb{C}$-isomorphism defined in such a way that
\[\xymatrix@C=4pc@R=1pc{Gr_{V}\Omega_f\ar[r]^-{s}_-{\simeq}&Gr_{V}\left(\dfrac{\mathbb{H}^{\prime \prime}_0}{\mathbb{H}^{\prime}_0}\right)\ar[r]^-{\Psi}_-{\simeq}&H^{n}(X_{\infty},\mathbb{C})\\
	gr_V[\eta_{j}]\;\ar@{|->}[r]&[s_{j}]\;\ar@{|->}[r]& A_j,} 
\]
for each $1\leq j\leq\mu$. Then we obtain the following result which has a graded flavor, where we recover a Varchenko's Lemma~\cite{varch2} that the maps $Gr_{V}\{f\}$ and $N$ have the same Jordan canonical normal form and also a consequence by looking at the bilinear form $Gr_{V} \,res_{f,0}({f}\bullet,\bullet)$ which is induced on the graded space $Gr_V\Omega_f$.
\begin{Cor}
	\label{cor:varchenkoslema}
	If we restrict to the graded space, with respect to the V-filtration, $Gr_{V}\big(\mathbb{H}^{\prime \prime}_0/\mathbb{H}^{\prime}_0\big)$ and consider the graded endomorphism $Gr_{V}\{f\}$ on $Gr_{V}\Omega_f$, then
	
	$$\big[Gr_{V}\{f\}\big]_{\underline{s}}=\big[N\big]_{\underline{A}},$$
	and for any $u,v\in H^{n}(X_{\infty},\mathbb{C})$:
	$$res_{f,0}\big(Gr_{V}\{f\}(\Psi \circ s)^{-1}(u)\,,\,(\Psi \circ s)^{-1}(v)\big)=B^{top}(u,v).$$
	where $\underline{s}$ is the basis of $Gr_{V}\big(\mathbb{H}^{\prime \prime}_0/\mathbb{H}^{\prime}_0\big)$ induced by $\{s_{l}\}_{1\leq l \leq\mu}$.
\end{Cor}
If the germ $f$ has finite monodromy, then $N=0$ and hence $B^{top}=0.$ Hence, we have next corollary which shows that the bilinear form $res_{f,0}(f\bullet,\bullet)$ has a little more information than the topological one given by the bilinear form $B^{top}.$ In the examples below we will illustrate this aspect.

\begin{Cor}
	If $f$ has finite monodromy, then
	\[res_{f,0}\bigg(f\omega\,,\,\eta\bigg)=B^{alg}\bigg(\varphi(\omega),\varphi(\eta)\bigg).\]
	Equivalently, for any $u,v\in H^{n}(X_{\infty},\mathbb{C})$:
	\begin{equation}\label{eqn:cor_res_BoplusB1}
		res_{f,0}\bigg(f\,\varphi^{-1}(u)\,,\,\varphi^{-1}(v)\bigg)=B^{alg}(u,v).
	\end{equation}
\end{Cor}

\section{Examples}
We will  give examples illustrating Theorem \ref{thm:firstpartthm} through Corollary \ref{eq:matrixMf} and, equations \eqref{eqn:B0}, \eqref{eqn:B1} and \eqref{eqn:ordenuno}; some calculations have been done using the computer algebra software SINGULAR~\cite{singular} 
 by means of libraries  {\bf gmssing.lib} \cite{MSch1} and {\bf mondromy.lib} \cite{MSch2}.

\begin{Ex}
 Let $f:(\mathbb{C}^{2},0)\lrar(\mathbb{C},0)$ be the germ of Isolated Hypersurface Singularity given by the semi-quasi-homogeneous polynomial $f=x^{5}+y^{6}+x^{4}y$; hence (using singular \cite{MSch1} with command {\bf monodromy}) we have that $M_f\neq 0$, but $M^{2}_{f}\equiv 0$; $f$ has finite monodromy, i.e., $N\equiv 0$. The Milnor number is $\mu=19$ and the spectrum $sp(f)$ is
  \[\xymatrix@R=1pc@C=.01pc{-5/8,&-11/24,&-5/12,&-7/24,&-1/4,&-5/24,&-1/8,&-1/12,&-1/24,&\\
  0,&1/24,&1/12,&1/8,&5/24,&1/4,&7/24,&5/12,&11/24,&5/8.
  }\]
Notice that $\alpha_{1}=-5/8$, $\alpha_{19}=5/8$ and, $\alpha_{i}\neq\alpha_{j}$, $i\neq j$ since $\alpha_{i}$ has multiplicity $d_{\alpha_{i}}=1$ and $\alpha_{i}<\alpha_{i+1}$, $\forall i=1,\ldots,18$.
In this example, $n=1$, hence the Saito-Hertling basis $\underline{h}$ for Brieskorne lattice $\Hpp$ is given by
\begin{align}
h_{1}&=s_{1}+\sum^{19}_{j=17}c^{(1)}_{1j} \;\dt s_{j},&& h_{2}= s_{2}+c^{(1)}_{2,19}\; \dt s_{19},&& h_{3}= s_{3}+c^{(1)}_{3,19}\; \dt s_{19},&&h_{j}= s_{j},\quad 4\leq j\leq 19.\nonumber
\end{align}
and one obtains a $\mathbb{C}$-basis $[\underline{h}]$ for $\Hpp/\dt^{-1}\Hpp$ so that $M_f$ in Corollary \ref{eq:matrixMf} has the matrix form:
\[[M_f]_{[\underline{\eta}]}=N_1\;\,(\textrm{since $N\equiv 0$}),\]
and
\begin{equation}\label{eqnExample1}
N_1=\begin{pmatrix}
 0&0& 0&0& \cdots&0\\
\hdotsfor[2]{6}\\
0&0& 0&0& \cdots&0\\
 \frac{1}{24}c^{(1)}_{1,17} &0&0&0&\cdots&0\\
  &&&&&\\
 \frac{2}{24}c^{(1)}_{1,18} &0 &0 & 0 & \cdots&0\\
  &&&&&\\
\frac{6}{24}c^{(1)}_{1,19}&\frac{2}{24}\,c^{(1)}_{1,18}&\frac{1}{24}\,c^{(1)}_{1,17}&0&\cdots&0\\
\end{pmatrix}_{19\times 19}
\end{equation}
which is symmetric with respect to the antidiagonal. The Tjurina number of $f$ is $\tau=17$, hence the rank of matrix (\ref{eqnExample1}) is $\mu-\tau=2$. Therefore at least one of the constants $c^{(1)}_{1,18}$ or $c^{(1)}_{1,17}$ is not equal to zero. 
Since the multiplicity of $\alpha_{10}=0=(n-1)/2$ is $d_{0}=1$, the $[\underline{\eta}]$-matrix expression for $res_{f,0}$ is
 \[
 [res_{f,0}]_{[\underline{\eta}]}= \begin{pmatrix}
  0&0&\dots&0&1\\
0&0&\dots&1&0\\
\hdotsfor[2]{5}\\
0&1&\dots&0&0\\
1&0&\dots&0&0
\end{pmatrix}_{19\times 19}
\]
 Finally, the $[\underline{\eta}]$-matrix for $res_{f,0}(f\bullet,\bullet):\Omega_f\times\Omega_f\lrar\mathbb{C}$ is
\begin{equation}\label{eqnExample1-2}
[res_{f,0}(M_f\bullet,\bullet)]_{[\underline{\eta}]}=\begin{pmatrix}
\frac{6}{24}c^{(1)}_{1,19}&\frac{2}{24}\,c^{(1)}_{1,18}&\frac{1}{24}\,c^{(1)}_{1,17}&0&\cdots&0\\
 &&&&&\\
 \frac{2}{24}c^{(1)}_{1,18} &0 &0 & 0 & \cdots&0\\
 &&&&&\\
 \frac{1}{24}c^{(1)}_{1,17} &0&0&0&\cdots&0\\
 0&0& 0&0& \cdots&0\\
\hdotsfor[2]{6}\\
0&0& 0&0& \cdots&0\\
\end{pmatrix}_{19\times 19}
\end{equation}
which is symmetric and has rank $2$. 
Since $n=1$, the involution $\kappa$ in \eqref{eqn:matrixQ} is given by
\[\kappa:\{1,\ldots,19\}\lrar \{1,\ldots,19\}\]
  \begin{displaymath}
k(j) = \left\{\begin{array}{ll}
20-j, & \textrm{if $\alpha_{j}\neq (n-1)/2=0$,}\\
        &  \\
10, & \textrm{if $\alpha_{10}=(n-1)/2=0$}.
\end{array} \right.
\end{displaymath}
More over the levels of each $s_{j}$ in the V-filtration flag are
\[r_{j}=\begin{cases}
1, \textrm{ for } 1/24,1/12,1/8,5/24,1/4,7/24,5/12,11/24,5/8,\\
0, \textrm{ for } -5/8,-11/24,-5/12,-7/24,-1/4,-5/24,-1/8,-1/12,-1/24,0
\end{cases}
\]
and the level for each $A_j$ in the Hodge flag is
\[p(j)=p(\alpha_{j})=\begin{cases}
0, \textrm{ for } 1/24,1/12,1/8,5/24,1/4,7/24,5/12,11/24,5/8\\
1, \textrm{ for } -5/8,-11/24,-5/12,-7/24,-1/4,-5/24,-1/8,-1/12,-1/24,0.
\end{cases}
\]
Hence, if we arrange the basis $\{A_j\}$ first integers and then the rest, 
\[S
=\begin{pmatrix}
1&0\\
0&\begin{pmatrix}
0&0&0&1\\
0&0&1&0\\
\hdotsfor[2]{4}\\
0&1&0&0\\
1&0&0&0
\end{pmatrix}_{18 \times 18}
\end{pmatrix}_{19\times 19}
\]
and
\[J=\frac{1}{2\pi i}\left(
  \begin{array}{ccccccccccc}
     -(\frac{1}{2\pi i})1 &&&&&&&&&&\\
    &1&0&&&&&&0&&\\
    &0&-1&&&&&&&&\\
    &&&1&0&&&&&&\\
    &&&0&-1&&&&&&\\
    &&&&&&\ddots&&&& \\
    &&&&&&&1&0&&\\
    &0&&&&&&0&-1&&\\
    &&&&&&&&&1&0 \\
    &&&&&&&&&0&-1 \\
  \end{array}
\right)=\begin{pmatrix}
\begin{pmatrix}
\frac{1}{(2\pi i)^{2}}(-1)
\end{pmatrix}&0\\
0&\frac{1}{(2\pi i)}\begin{pmatrix}
I&0\\
0&-I
\end{pmatrix}
\end{pmatrix}
\]
Hence,
\[Q=(2\pi i)^{1}\left(
  \begin{array}{ccccccccccc}
-(2\pi i)1&&&&&&&&&&\\
             & 0&-1&&&&&&0&&\\
    		&1&0&&&&&&&&\\
    		&&&0&-1&&&&&&\\
    		&&&1&0&&&&&&\\
    		&&&&&&\ddots&&&&\\
    		&&&&&&&0&-1&&\\
    		&&&0&&&&1&0&&\\
    		&&&&&&&&&0&-1\\
    		&&&&&&&&&1&0\\
  \end{array}
\right)=\begin{pmatrix}
\begin{pmatrix}
(2\pi i)^{2}(-1)
\end{pmatrix}&0\\
0&(2\pi i)\begin{pmatrix}
0&-I\\
I&0
\end{pmatrix}
\end{pmatrix}
\]
It verifies that $S=QJ$. Since $N\equiv 0$, the bilinear form $Q(N\bullet,J\bullet)$ is trivial and we verify the formula $res_{f,0}(M_f\bullet,\bullet)=Q(\mathbf{N}_1,J)$. 
\end{Ex}

\begin{Ex}
Consider the germ of Isolated Hypersurface Singularity given by the semi-quasi-homogeneous polynomial $f=x^{4}+y^{5}+xy^{4}$, for which we can check (using singular \cite{MSch1} with command {\bf monodromy }) that $N\equiv 0$. And that the corresponding Milnor and Tjurina number are $\mu=12$ and $\tau=11$; hence the rank of $M_f$ is $rank(M_f)=1$ and $M_f\neq 0$. Since $n=1$, $M^{2}_f\equiv 0$; we also may compute
the spectrum $sp(f)$:
  \[\xymatrix@R=1pc@C=.01pc{\alpha_1=-11/20,&-7/20,&-3/10,&-3/20,&-1/10,&-1/20,&\\
  1/20,&1/10,&3/20,&3/10&7/20,&\alpha_{12}= 11/20.
  }\]
Notice that each $\alpha_{i}$ has multiplicity $d_{\alpha_{i}}=1$ and $\alpha_{i}<\alpha_{i+1}\;\;\forall i=1,\ldots,11$.
The Saito-Hertling basis $\underline{h}$ for Brieskorne lattice $\Hpp$ is given by
\begin{align}
h_{1}&=s_{1}+c^{(1)}_{1,12} \;\dt s_{12},\hspace{2 cm} h_{j}= s_{j},\quad 2\leq j\leq 12.\nonumber
\end{align}
In the corresponding $\mathbb{C}$-basis $[\underline{h}]$ for $\Hpp/\dt^{-1}\Hpp$, $M_f$ has the matrix form: 
\[
[M_f]_{[\underline{\eta}]}=N_1=\begin{pmatrix}
 0&0& 0&0& \cdots&0\\
\hdotsfor[2]{6}\\
0&0& 0&0& \cdots&0\\
 0 &0&0&0&\cdots&0\\
0 &0 &0 & 0 & \cdots&0\\
\frac{1}{10} c^{(1)}_{1,12}&0&0&0&\cdots&0\\
\end{pmatrix}_{12\times 12}
\]
the $[\underline{\eta}]$-matrix expression for $res_{f,0}$ is
 \[
 [res_{f,0}]_{[\underline{\eta}]}= \begin{pmatrix}
  0&0&\dots&0&1\\
0&0&\dots&1&0\\
\hdotsfor[2]{5}\\
0&1&\dots&0&0\\
1&0&\dots&0&0
\end{pmatrix}_{12\times 12}
\]
 Finally, the $[\underline{\eta}]$-matrix for $res_{f,0}(M_f\bullet,\bullet):\Omega_f\times\Omega_f\lrar\mathbb{C}$ is
\[
[res_{f,0}(M_f\bullet,\bullet)]_{[\underline{\eta}]}=\begin{pmatrix}
\frac{1}{10}c^{(1)}_{1,12}&&0&0&\cdots&0\\
 0&& 0&0& \cdots&0\\
\hdotsfor[2]{6}\\
0&& 0&0& \cdots&0\\
\end{pmatrix}_{12\times 12}
\]
which is symmetric and has rank $1$. 
Since $n=1$ and $\alpha\neq (n-1)/2=0$ for all $\alpha \in sp(f)$, the involution $\kappa:\{1,\ldots,12\}\lrar \{1,\ldots,12\}$ is given by $k(j) = 13-j$.

More over the levels of each $s_{j}$ in the V-filtration flag are
\[r_{j}=\begin{cases}
1, \textrm{ for }   1/20,1/10,3/20,3/10,7/20,11/20\\
0, \textrm{ for } -11/20,-7/20,-3/10,-3/20,-1/10,-1/20
\end{cases}
\]
and the level for each $A_j$ in the Hodge flag is
\[p(j)=p(\alpha_{j})=\begin{cases}
0, \textrm{ for }   1/20,1/10,3/20,3/10,7/20,11/20\\
1, \textrm{ for } -11/20,-7/20,-3/10,-3/20,-1/10,-1/20
\end{cases}
\]
Notice that  for any $\alpha\in sp(f)$, $\alpha\neq 0$. Hence, $H^{1}(X_{\infty},\mathbb{C})=H^{1}(X_{\infty},\mathbb{C})_{\neq 1}$; up to arranging basis $\{A_j\}$, 
\[S= \begin{pmatrix}
  0&0&\dots&0&1\\
0&0&\dots&1&0\\
\hdotsfor[2]{5}\\
0&1&\dots&0&0\\
1&0&\dots&0&0
\end{pmatrix}_{12\times 12}
\]
and
\[J=\bigg(\frac{1}{2\pi i}\bigg)\left(
  \begin{array}{ccccccccccc}
    1&0&&&&&&0&&\\
    0&-1&&&&&&&&\\
    &&1&0&&&&&&\\
    &&0&-1&&&&&&\\
    &&&&&\ddots&&&& \\
    &&&&&&1&0&&\\
    0&&&&&&0&-1&&\\
    &&&&&&&&1&0 \\
    &&&&&&&&0&-1 \\
  \end{array}
\right)=\bigg(\frac{1}{2\pi i}\bigg)\begin{pmatrix}
I&0\\
0&-I
\end{pmatrix}
\]
Hence,
\[
Q=(2\pi i)\begin{pmatrix}
0&-1&&&&&&0&&\\
1&0\\
&&0&-1\\
&&1&0&\\
&&&&&\ddots\\
&&&&&&0&-1\\
&&0&&&&1&0\\
&&&&&&&&0&-1\\
&&&&&&&&1&0\\
\end{pmatrix}=(2\pi i)\begin{pmatrix}
0&-I\\
I&0
\end{pmatrix}
\]
It verifies that $S=QJ$ and $res_{f,0}(M_f\bullet,\bullet)=Q(\mathbf{N}_1\bullet,J\bullet)$ since the bilinear form $Q(N\bullet,J\bullet)$ is trivial.
\end{Ex}

\begin{Ex}
  Let $f:(\mathbb{C}^{2},0)\lrar(\mathbb{C},0)$ be the germ of Isolated Hypersurface Singularity given by the not semi-quasi-homogeneous polynomial $f=x^{5}+y^{5}+x^{2}y^{2}$. The Milnor and Tjurina numbers are $\mu= 11$ and $\tau=10$; hence the rank of  $M_f$ is $rk(M_f)=1$ and $M_f\neq 0$, but $M^{2}_f\equiv 0$ since $n=1$.  It can be checked (using Singular \cite{MSch2} with command {\bf monodromy}) that $f$ has a non finite monodromy, more over, $N:H^{1}(X_{\infty},\mathbb{C})\lrar H^{1}(X_{\infty},\mathbb{C})$ has a Jordan block of size $2\times 2$ corresponding to eigenvalue $\lambda=-1$, and then $N\neq 0$, but $N^{2}\equiv 0$ by the Monodromy Theorem. The spectrum $sp(f)$ is:
  \[\xymatrix@R=1pc@C=.01pc{(\alpha,d_{\alpha}):(-1/2,1),(-3/10,2),(-1/10,2),(0,1),(1/10,2),(3/10,2),(1/2,1).\\
  }\]
Here $\alpha_{1}=-1/2$ and $\alpha_{11}=1/2$. Each $\alpha_{i}\neq-1/2,1/2$ has multiplicity $d_{\alpha_{i}}=2$ and $d_{\alpha_{11}}=d_{\alpha_{1}}=1$. Notice that $\alpha_{1}+1=\alpha_{11}$ and $\alpha_{i}+1\notin sp(f)\subset(-1,1)$. Hence, the Saito-Hertling basis for Brieskorn lattice $\Hpp$ is such that
 \begin{align*}
 &h_{i}=s_{i}\,, \;1\leq i\leq \mu,&\\
&h_{11}=h_{\nu(1)}=s_{\nu(1)}\neq 0,\hskip 1cm h_{\nu(i)}=0,\;\; 2\leq i\leq 11,&\\
&th_{i}=(\alpha_{i}+1)\dt^{-1}h_{i}+h_{\nu(i)}, 1\leq i \leq 11,&\\
&[th_{i}] = \left\{ \begin{array}{ll}
[h_{\nu(1)}]=[h_{11}]\in\Hpp/\dt^{-1}\Hpp\\
& \\
0,\textrm{if $i\neq 1$}\,.
\end{array} \right.
\end{align*}
Hence, the  $[\underline{\omega}]$-matrix for $M_f$ is given by
 \[
 \big[M_f\big]_{[\underline{\eta}]}= \begin{pmatrix}
0&0&\dots&0\\
\hdotsfor[2]{4}\\
0&0&\dots&0\\
1&0&\dots&0
\end{pmatrix}_{11\times 11}=[N]_{\underline{A}}
\]
Here $N_1\equiv 0$ and one has the trivial bilinear form $N^{tr}_1SJ=0$.
The involution $\kappa_f$ is given by
\[\kappa_f: \{1,\ldots,11\}\lrar \{1,\ldots,11\},\hskip 1 cm \kappa_{f}(i)=12-i.\]
and
 \[
 [res_{f,0}]_{[\underline{\eta}]}= \begin{pmatrix}
  0&0&\dots&0&1\\
0&0&\dots&1&0\\
\hdotsfor[2]{5}\\
0&1&\dots&0&0\\
1&0&\dots&0&0
\end{pmatrix}_{11\times 11}
\]
since the multiplicity of $\alpha_{6}=0=(n-1)/2$ is $d_{0}=1$.

On the other hand, the levels of each $s_{j}$ in the V-filtration flag are
\[r_{j}=\begin{cases}
1, \textrm{ for } (1/10,2),(3/10,2),(1/2,1)\\
0, \textrm{ for } (-1/2,1),(-3/10,2),(-1/10,2),(0,1)
\end{cases}
\]
and the level for each $A_j$ in the Hodge flag is
\[p(j)=p(\alpha_{j})=\begin{cases}
0, \textrm{ for } (1/10,2),(3/10,2),(1/2,1)\\
1, \textrm{ for } (-1/2,1),(-3/10,2),(-1/10,2),(0,1)
\end{cases}
\]
Hence, if we arrange the basis $\{A_j\}$ first integers and then the rest, 
\[S
=\begin{pmatrix}
1&0\\
0&\begin{pmatrix}
0&0&0&1\\
0&0&1&0\\
\hdotsfor[2]{4}\\
0&1&0&0\\
1&0&0&0
\end{pmatrix}_{10 \times 10}
\end{pmatrix}_{11\times 11}
\]
and
\[J=\frac{1}{2\pi i}\left(
  \begin{array}{ccccccccccc}
     -(\frac{1}{2\pi i})1 &&&&&&&&&&\\
    &1&0&&&&&&0&&\\
    &0&-1&&&&&&&&\\
    &&&1&0&&&&&&\\
    &&&0&-1&&&&&&\\
    &&&&&&\ddots&&&& \\
    &&&&&&&1&0&&\\
    &0&&&&&&0&-1&&\\
    &&&&&&&&&1&0 \\
    &&&&&&&&&0&-1 \\
  \end{array}
\right)=\begin{pmatrix}
\begin{pmatrix}
\frac{1}{(2\pi i)^{2}}(-1)
\end{pmatrix}&0\\
0&\frac{1}{(2\pi i)}\begin{pmatrix}
I&0\\
0&-I
\end{pmatrix}
\end{pmatrix}
\]
Hence,
\[Q=(2\pi i)^{1}\left(
  \begin{array}{ccccccccccc}
-(2\pi i)1&&&&&&&&&&\\
             & 0&-1&&&&&&0&&\\
    		&1&0&&&&&&&&\\
    		&&&0&-1&&&&&&\\
    		&&&1&0&&&&&&\\
    		&&&&&&\ddots&&&&\\
    		&&&&&&&0&-1&&\\
    		&&&0&&&&1&0&&\\
    		&&&&&&&&&0&-1\\
    		&&&&&&&&&1&0\\
  \end{array}
\right)=\begin{pmatrix}
\begin{pmatrix}
(2\pi i)^{2}(-1)
\end{pmatrix}&0\\
0&(2\pi i)\begin{pmatrix}
0&-I\\
I&0
\end{pmatrix}
\end{pmatrix}
\]
It verifies that $S=QJ$ and $res_{f,0}(M_f\bullet,\bullet)=Q(N\bullet ,J\bullet)$ since the bilinear form $Q(\mathbf{N}_1,J)$ is trivial.
\end{Ex}

\begin{Ex}
We will consider the M. Saito example as in~\cite[p.18 ]{MSaito2} which was stated by using the Tom-Sebastiani type theorem (\cite{SherkSteen}). Set
\[f=g+g^{\prime} \textrm{ with } g=x^{10}+y^{3}+x^{2}y^{2},\;g^{\prime}=z^{6}+w^{5}+z^{4}w^{3}.
\]
Using Singular~\cite{MSch1} we may compute that Milnor and Tjurina numbers are $\mu=280$ and $\tau=248$, and the spectrum $sp(f)$ is:
\begin{align*}
(\alpha,d_{\alpha}):
(-2/15,1),(-1/30,1),\\
(1/30,1),(1/15,2),(2/15,1),(1/6,2),(1/5,2),(7/30,2),(4/15,3),(3/10,1),(1/3,2),(11/30,6),\\
(2/5,3),(13/30,3),(7/15,5),(1/2,2),(8/15,7),(17/30,8),(3/5,4),(19/30,5),(2/3,6),(7/10,6),\\
(11/15,9),(23/30,9),(4/5,5),(5/6,6),(13/15,10),(9/10,7),(14/15,10),(29/30,9),\\
(1,4),(31/30,9),(16/15,10),(11/10,7),(17/15,10),(7/6,6),(6/5,5),(37/30,9),(19/15,9),\\
(13/10,6),(4/3,6),(41/30,5),(7/5,4),(43/30,8),(22/15,7),(3/2,2),(23/15,5),(47/30,3),(8/5,3),\\
(49/30,6),(5/3,2),(17/10,1),(26/15,3),(53/30,2),(9/5,2),(11/6,2),(28/15,1),(29/15,2),(59/30,1),\\
(61/30,1),(32/15,1)
\end{align*}
The rank of $M_f$ is $rk\;M_f=32$; with Singular \cite{MSch1,MSch2} (with command {\bf jacoblift }) it can be checked that $M_f\neq 0$, $M_f^{2}\neq 0$, but $M_f^{3}\equiv 0$; and (with command {\bf monodromy}) one obtains that $N$ only has Jordan blocks of size $1\times 1$ and several Jordan bocks of (maximal) size $2\times 2$; hence $N\neq 0$, $N^{2}\neq 0$, but $N^{3}\equiv 0$. Our interest in the present example is the following M. Saito result~\cite{MSaito2}. Let $\Hpp(f),\Hpp(g),\Hpp(g^{\prime})$ be the Brieskorn lattices of $f,g,g^{\prime}$, respectively. Hence, there are canonical isomorphisms (cf. \cite{SherkSteen})
\[\Hpp(f)\simeq\Hpp(g)\otimes_{R}\Hpp(g^{\prime}),\hskip 1cm \frac{\Hpp(f)}{\dt ^{-1}\Hpp(f)}\simeq\frac{\Hpp(g)}{\dt ^{-1}\Hpp(g)}\otimes_{\mathbb{C}}\frac{\Hpp(g^{\prime})}{\dt ^{-1}\Hpp(g^{\prime})}, \]
($R:=\mathbb{C}\{\{\dt ^{-1}\}\}$ is the ring of micro-differential operators with constant coefficients) such that the action of $t$ on the left hand side is identified with $t\otimes Id+Id\otimes t$. There is a subspace  of rank $ 6$, denoted by $\Hpp(f)^{\prime}\subset\Hpp(f)$, such that the $[\underline{h}]$-matix of $t$ on the corresponding class subset $\overline{\Hpp(f)^{\prime}}\subset \Hpp(f)/\dt ^{-1}\Hpp(f)$ is given by
\[\big[t\big]_{|[\underline{h}]}=\begin{pmatrix}
0&0&0&0&0&0\\
0&0&0&0&0&0\\
1&0&0&0&0&0\\
a&0&0&0&0&0\\
0&0&0&0&0&0\\
0&0&a&1&0&0\\
  \end{pmatrix}=[N]+N_1\]
with $a\in\mathbb{C}\setminus\{0\}$. As a consequence, the $[\underline{\eta}]$-matrix for $M_f$ is also $[N]+N_1$. Notice that if we restrict to $\ker(N)$, the bilinear form $Q(\mathbf{N}_1,J)$ is non trivial on $\varphi^{-1}(\ker(N))\cap (\overline{\Hpp(f)^{\prime}})$.
\end{Ex}

\section{conclusions}

Making use of the Saito-Hertling basis which is constructed from the Deligne spliting associated to the Steenbrink-Hertling Polarized Mixed Hodge Structure with respect to  polarization $Q$, we may choose a Jordan basis for vanishing canonical cohomology fiber $H^{n}(X_{\infty},\mathbb{C})$ adapted to the weight filtration $W(N)$, in such a way that we can describe a normal form for the bilinear forms and maps:
\begin{itemize}
	\item Endomorphisms on $H^{n}(X_{\infty},\mathbb{C})$: $N$ and $\mathbf{N}_1$;
	\item Map muiltiplication by $f$ on $\Omega^f$: $M_f$;
	\item The isomophism $\varphi$ from $\Omega^f$ to $H^{n}(X_{\infty},\mathbb{C})$;
	\item The automorphism on $H^{n}(X_{\infty},\mathbb{C})$: $J$;
	\item Polarization on $H^{n}(X_{\infty},\mathbb{C})$: $Q(\bullet,\bullet)$;
	\item The Lefschetz bilinear forms with topological setting on $H^{n}(X_{\infty},\mathbb{C})$: $Q(N^{\ell}\bullet,\bullet), \ \  0\leq \ell \leq n$;
	\item The topological weighted bilinear forms on $H^{n}(X_{\infty},\mathbb{C})$: $Q(N^{\ell}\bullet,J\bullet),\ \ 0 \leq \ell \leq n$;
	\item The algebraic bilinear forms on $H^{n}(X_{\infty},\mathbb{C})$: $Q(\mathbf{N}_1^{\ell}\bullet,J\bullet),\ \ 0 \leq \ell \leq n$;
	\item The Grothendieck paring: $res_{f,0}(\bullet,\bullet)$;
	\item The higher bilinear forms on the Jacobian module $\Omega_f$: $res_{f,0}(M_f^{j}\bullet,\bullet),\ \ 0\leq j\leq n$.
\end{itemize}
Our main result given by Theorem \ref{thm:firstpartthm}, uses these normal forms to relate the higher bilinear forms in $\Omega_f$ to the weighted topological and algebraic higher bilinear forms on $H^{n}(X_{\infty},\mathbb{C})$.
Essentially, these results are produced using the normal form that the Grothendieck pairing inherits from the Saito-Hertling basis for the Brieskorn lattice. Hence we can note that, the bilinear forms in $\Omega_f$ have more information than the weighted topological which is detected assuming that the germ has finite monodromy. The interesting part arises from the fact that such weights are determined by the authomorphism $J$ which is constructed from the Hodge filtration by means of the spectrum of singularity germ $f$. 

Besides, our results also state that the Jordan chains of $M_f$ can be obtained as a binding of topological $N$-spectral chains, without taking the grading that the $V$-filtration produces.

The interest in the analysis of the bilinear forms in the Milnor algebra (or the Jacobian module) arises from the attempt of the second author and collaborators \cite{Giraldo-GM} for the understanding of it geometrical meaning, since there is a relationship between the signature of these higher bilinear forms to indices of vector fields whenever the germ $f$ also is real analytic.  

Finally, it is worth mentioning that in \cite{DelaRosa2018} using a weaker approach and without using the Saito-Hertling basis, the author shows that  the bilinear forms $res_f(f^j\bullet ,\bullet)$ have an additive expansion in terms of the bilinear forms $Q(N^j\bullet ,\bullet)$. Such additive expansions depend only on the asymptotic expansions for elements on the Jacobian module, which are induced by the $V$-filtration.  
 

 \end{document}